%% file: ms.tex
\documentclass[11pt, a4paper]{article}
\usepackage{a4wide}

\usepackage[T1]{fontenc}
\usepackage{lmodern}
\usepackage{textcomp}
\usepackage[utf8]{inputenc}
\usepackage[english]{babel}
\usepackage[intlimits]{amsmath}
\usepackage{amssymb}
\usepackage[parfill]{parskip} 

\usepackage[shortlabels]{enumitem}
\usepackage{amsthm}
\begingroup
    \makeatletter
    \@for\theoremstyle:=definition,remark,plain\do{%
        \expandafter\g@addto@macro\csname th@\theoremstyle\endcsname{%
            \addtolength\thm@preskip\parskip
            }%
        }
\endgroup

\usepackage{booktabs}
\usepackage{subfig}
\usepackage{placeins}
\usepackage{graphicx}
\usepackage{sidecap}
\usepackage{datetime}

\newdateformat{mydate}{\monthname[\THEMONTH], \THEYEAR}
\newcommand{\pder}[2]{\frac{\partial #1}{\partial #2}}

\newcommand{\norm}[1]{\left\Vert #1\right\Vert}

\newcommand{\normTwo}[1]{\left\Vert #1\right\Vert^2}

\newcommand{\innerL}[2]{\left(#1, #2\right)_{L^2\left(\Omega\right)}}
\newcommand{\inner}[2]{\left(#1, #2\right)}
\newcommand{\dualprod}[2]{\left\langle #1, #2\right\rangle}

\newcommand{\vertiii}[1]{{\left\vert\kern-0.25ex\left\vert\kern-0.25ex\left\vert #1 
    \right\vert\kern-0.25ex\right\vert\kern-0.25ex\right\vert}_{\epsilon}}
\newcommand{\vertiiih}[1]{{\left\vert\kern-0.25ex\left\vert\kern-0.25ex\left\vert #1 
    \right\vert\kern-0.25ex\right\vert\kern-0.25ex\right\vert}_{\epsilon,h}}

\newcommand{\Schur}{\operatorname{Schur}}

\theoremstyle{plain}

\newtheorem{theorem}{Theorem}[section]
\newtheorem{corollary}{Corollary}[section]
\newtheorem{lemma}{Lemma}[section]
\newtheorem{problem}{Problem}[section]
\newtheorem{remark}{Remark}[section]

\newcommand{\foralls}{\forall\,}
\newcommand{\lTwo}{L^2\left(\Omega\right)}
\newcommand{\lTwoB}{L^2(\partial\Omega)}

\title{Schur complement preconditioners for multiple saddle point problems of block tridiagonal form with application to optimization problems} 

\author{
	Jarle Sogn\footnote{Institute of Computational Mathematics, Johannes Kepler University Linz, 4040 Linz, Austria (jarle@numa.uni-linz.ac.at, zulehner@numa.uni-linz.ac.at). The research was supported by the Austrian Science Fund (FWF): S11702-N23.} \, and
	Walter Zulehner\footnotemark[1]
}

\date{}

\begin{document}
\maketitle


\begin{abstract}
The importance of Schur complement based preconditioners are well-established for classical saddle point problems in $\mathbb{R}^N \times \mathbb{R}^M$. In this paper we extend these results to multiple saddle point problems in Hilbert spaces $X_1\times X_2 \times \cdots \times X_n$. 
For such problems with a block tridiagonal Hessian and a well-defined sequence of associated Schur complements, sharp bounds for the condition number of the problem are derived which do not depend on the involved operators. These bounds can be expressed in terms of the roots of the difference of two Chebyshev polynomials of the second kind.
If applied to specific classes of optimal control problems the abstract analysis leads to new existence results as well as to the construction of efficient preconditioners for the associated discretized optimality systems.

\end{abstract}

\FloatBarrier
\input{Introduction}
\input{Theory}
\input{OCPoisson}
\input{Discretization}
\input{ResultsBoth}
\input{Discussion}

\appendix
\input{Appendix}

\bibliographystyle{abbrv}
\bibliography{bibliography}

\end{document}

%% file: Introduction.tex
\section{Introduction}

In this paper we discuss the well-posedness of a particular class of saddle point problems in function spaces and the related topic of efficient preconditioning of such problems after discretization. Problems of this class arise as the optimality systems of optimization problems in function spaces with a quadratic objective functional and constrained by linear partial differential equations. Another source for such problems are mixed formulations of elliptic boundary value problems. For numerous applications of saddle point problems we refer to the seminal survey article \cite{benzi2005numerical}.

To be more specific, we consider saddle point problems of the following form:
For a given functional $\mathcal{L}(x_1,x_2,\ldots,x_n)$ defined on a product space $X_1 \times X_2 \times \cdots \times X_n$  of Hilbert spaces $X_i$ with $n \ge 2$, find an $n$-tuple $(x_1^*,x_2^*,\ldots,x_n^*)$ from this space such that its component $x_i^*$ minimizes $\mathcal{L}^{[i]}(x_i)$ for all odd indices $i$ and maximizes $\mathcal{L}^{[i]}(x_i)$ for all even indices $i$, where
$\mathcal{L}^{[i]}(x_i) = \mathcal{L}(x_1^*,\ldots,x_{i-1}^*,x_i,x_{i+1}^*,\ldots,x_n^*)$.

Very often the discussion of saddle point problems is restricted to the case $n=2$. We will refer to these problems as classical saddle point problems.
In this paper we are interested in the general case $n\ge 2$. We call such problems multiple saddle point problems.
Saddle point problems with $n=3$ and $n=4$ are typically addressed in literature as double (or twofold) and triple (or threefold) saddle point problems, respectively. 

For notational convenience $n$-tuples $(x_1,x_2,\ldots,x_n) \in X_1 \times X_2 \times \cdots \times X_n$ are identified with the corresponding column vectors  $\mathbf{x}=(x_1,x_2,\ldots,x_n)^\top$ from the corresponding space $\mathbf{X}$. 
We consider only linear problems; that is, we assume that
\[
  \mathcal{L}(\mathbf{x}) 
    = \frac12 \dualprod{\mathcal{A} \mathbf{x}}{\mathbf{x}} - \dualprod{\mathbf{b}}{\mathbf{x}},
\]
where $\mathcal{A}$ is a bounded and self-adjoint linear operator which maps from $\mathbf{X}$ to its dual space $\mathbf{X}'$, $\mathbf{b} \in \mathbf{X}'$, and $\dualprod{.}{.}$ denotes the duality product. Observe that $\mathcal{A}$ is the (constant) Hessian of $\mathcal{L}(\mathbf{x})$ and has a natural $n$-by-$n$ block structure consisting of elements $\mathcal{A}_{ij}$ which map from $X_j$ to $X_i'$. 

The existence of a saddle point necessarily requires that the block diagonal elements $\mathcal{A}_{ii}$ are positive semi-definite for odd indices $i$ and negative semi-definite for even indices $i$. Under this assumption the problem of finding a saddle point of $\mathcal{L}$ is equivalent to find a solution $\mathbf{x}^* \in \mathbf{X}$ of the linear operator equation
\begin{equation} \label{Axb}
  \mathcal{A} \mathbf{x} = \mathbf{b}.
\end{equation}
Typical examples for the case $n=2$ are optimality systems of constrained quadratic optimization problems, where $\mathcal{L}$ is the associated Lagrangian, $x_1$ is the primal variable, and $x_2$ is the Lagrangian multiplier associated to the  constraint. Optimal control problems viewed as constrained optimization problems fall also into this category with $n=2$. However, since in this case the primal variable itself consists  of two components, the state variable and the control variable, we can view such problems also as double saddle problems (after some reordering of the variables), see \cite{KAMmagneOC} and Chapter \ref{sec:OCPoissonProblem}. For an example of a triple saddle point problem, see, e.g., \cite{langer2007inexact}. Other sources of multiple saddle point problems can be found, e.g., in \cite{GaticaHeuer2000}, \cite{GaticaHeuer2002}, \cite{Gatica2007}, \cite{Benzi2017} and the references therein.

The goal of this paper is to extend well-established results on block diagonal preconditioners for classical saddle point problems in $\mathbb{R}^N \times \mathbb{R}^M$ as presented in \cite{kuznetsov95}, \cite{murphy2000note} to multiple saddle point problems in Hilbert spaces. This goal is achieved for operators $\mathcal{A}$ of block tridiagonal form which possess an associated sequence of positive definite Schur complements. We will show for a particular norm build from these Schur complements that the condition number of the operator $\mathcal{A}$ is bounded by a constant independent of $\mathcal{A}$.  So, if $\mathcal{A}$ contains any sensitive model parameters (like a regularization parameter) or $\mathcal{A}$ depends on some discretization parameters (like the mesh size), the bound of the condition number is independent of these quantities. This, for example, prevents the performance of iterative methods from deteriorating for small regularization parameters or small mesh sizes. Moreover we will show that the bounds are solely given in terms of the roots of the difference of two Chebyshev polynomials of the second kind and that the bounds are sharp for the discussed class of block tridiagonal operators.

The abstract analysis allows to recover known existence results under less restrictive assumptions. This was the main motivation for extending the analysis to Hilbert spaces. We will exemplify this for optimal control problems with a second-order elliptic state equation, distributed observation, and boundary control, as discussed, e.g., in \cite{may2013error}, and for boundary observation and distributed control, as discussed, e.g., in \cite{KAMmagneOC}. Another outcome of the abstract analysis is the construction of preconditioners for discretized optimality systems which perform well in combination with Krylov subspace methods for solving the linear system. Here we were able to recover known results from \cite{KAMmagneOC} and extend them to other problems. The article \cite{KAMmagneOC} was very influential for the study presented here. As already noticed in \cite{KAMmagneOC}, there is a price to pay for the construction of these efficient preconditioners: For second-order elliptic state equations discretization spaces of continuously differentiable functions are required, for which we use here  technology from Isogeometric Analysis (IgA), see the monograph \cite{cottrell2009isogeometric}.

Observe that the analysis presented here is valid for any number $n\ge 2$ of blocks. There are numerous contributions for preconditioning classical saddle point problems, see \cite{benzi2005numerical} and the references cited there. See, in particular, among many others contributions \cite{Pearson2012}, for Schur complement based approaches.
For other results on the analysis and the construction of preconditioners for double/twofold and triple/threefold saddle point problems see, e.g., \cite{GaticaHeuer2000}, \cite{GaticaHeuer2002}, \cite{Gatica2007},\cite{langer2007inexact}, \cite{Pestana2016}, \cite{Benzi2017}.

The paper is organized as follows. The abstract analysis of a class of multiple saddle point problems of block tridiagonal form is given in Section 2. Section 3 deals with the application to particular optimization problems in function spaces. Discretization and efficient realization of the preconditioner are discussed in Section 4. A few numerical experiments are shown in Section 5 for illustrating the abstract results. The paper ends with some conclusions in Section 6 and an appendix, which contains some technical details related to Chebyshev polynomial of the second kind used in the proofs of the abstract results in Section 2.

%% file: Theory.tex
\section{Schur complement preconditioners}
\label{sec:Theory}

The following notations are used throughout the paper.
Let $X$ and $Y$ be Hilbert spaces with dual spaces $X'$ and $Y'$. 
For a bounded linear operator $B: X \rightarrow Y'$, its Banach space adjoint $B': Y \rightarrow X'$ is given by
\begin{equation*}
  \langle B'y,x\rangle = \langle Bx,y\rangle
  \quad \text{for all} \ x\in X, \ y \in Y,
\end{equation*}
where $\left\langle \cdot, \cdot\right\rangle$ the denotes the duality product.
For a bounded linear operator $L: X \rightarrow Y$, its Hilbert space adjoint $L^{*}: Y \rightarrow X$ is given by 
\begin{equation*}
\inner{L^{*}y}{x}_{X} = \inner{Lx}{y}_{Y}
  \quad \text{for all} \ x\in X, \ y \in Y,
\end{equation*}
where $\inner{\cdot}{\cdot}_X$ and $\inner{\cdot}{\cdot}_Y$ are the inner products of $X$ and $Y$ with associated norms $\|\cdot\|_X$ and $\|\cdot\|_Y$, respectively.

Let $X = U \times V$ with Hilbert spaces $U$ and $V$. Then its dual $X'$ can be identified with $U' \times V'$. For a linear operator $T : U \times V \rightarrow U' \times V'$ of a 2--by--2 block form
\[
  T 
    = \begin{pmatrix}
       A & C \\ B & D
      \end{pmatrix},
\]
with an invertible operator $A : U \rightarrow U'$, its Schur complement $\Schur T : V \rightarrow V'$ is given by
\begin{equation*}
  \Schur T = D - B A^{-1} C .
\end{equation*}

With these notations we will now precisely formulate the assumptions on problem (\ref{Axb}) as already indicated in the introduction. 
Let $\mathbf{X} = X_1 \times X_2\times \ldots \times X_n$ with Hilbert spaces $X_i$ for $i=1,2,\ldots,n$, and let the linear operator $\mathcal{A}: \mathbf{X} \rightarrow \mathbf{X}'$ be of $n$--by--$n$ block tridiagonal form
  \begin{equation}
    \label{eq:ADefNN}
    \mathcal{A} =     
    \begin{pmatrix} 
      A_1 & B_1'  &   & \\
      B_1 & -A_2  & \ddots &  \\
          &\ddots & \ddots & B_{n-1}' \\[1ex]
          &       & B_{n-1}   & (-1)^{n-1}A_n
  \end{pmatrix},
  \end{equation}
where $A_i:X_i \rightarrow X_i'$, $B_i: X_i \rightarrow X_{i+1}'$ are bounded operators, and, additionally, $A_i$ are self-adjoint and positive semi-definite; that is,
\begin{equation} \label{spsd}
  \langle A_i y_i,x_i \rangle = \langle A_i x_i,y_i \rangle
  \quad \text{and} \quad 
  \langle A_i x_i, x_i \rangle \ge 0
  \quad \foralls \ x_i,\ y_i \in X_i.
\end{equation}
The basic assumption of the whole paper is now that the operators $\mathcal{A}_i$ consisting of the first $i$ rows and columns of $\mathcal{A}$ are invertible operators from $X_1 \times \cdots \times X_i$ to $X_1' \times \cdots \times X_i'$. That allows to introduce the linear operators 
\begin{equation*}
  S_{i+1} = (-1)^i \, \Schur \mathcal{A}_{i+1}
  \quad \text{for}\quad i=1,\ldots,n-1,
\end{equation*}
where, for the definition of the Schur complement, $\mathcal{A}_{i+1}$ is interpreted as the 2--by-2 block operator
\begin{equation*}
  \mathcal{A}_{i+1}
    = \begin{pmatrix} 
      \mathcal{A}_i & \mathbf{B}_i' \\ 
      \mathbf{B}_i & (-1)^i \, A_{i+1}
      \end{pmatrix}, 
      \quad \text{where} \quad \mathbf{B}_i = \begin{pmatrix} 0 & \ldots & 0 & B_i \end{pmatrix} .
\end{equation*}
It is easy to see that
\begin{equation}
\label{eq:SchurCompCom}
S_{i+1} = A_{i+1} + B_i S^{-1}_i B_i', \quad \text{for}\quad i = 1, 2,\ldots,n-1.
\end{equation}
with initial setting $S_1 = A_1$. 

The following basic result holds:
\begin{lemma} \label{lem:one}
Assume that 
$A_i:X_i \rightarrow X_i'$, $i = 1,2\ldots,n$ are bounded operators satisfying (\ref{spsd}), 
$B_i: X_i \rightarrow X_{i+1}'$, $i=1,\ldots,n-1$ are bounded operators, and 
$S_i$, $i=1,2,\ldots,n$, given by (\ref{eq:SchurCompCom}), are well-defined and positive definite, that is 
\begin{equation*}
  \langle S_i x_i, x_i \rangle \ge \sigma_i \, \|x_i\|_{X_i}^2
  \quad \foralls \ x_i \in X_i,
\end{equation*}
for some positive constants $\sigma_i$. Then $\mathcal{A}$ is an isomorphism form $\mathbf{X}$ to $\mathbf{X}'$.
\end{lemma}
\begin{proof}\unskip
From the lemma of Lax-Milgram it follows that $S_i$, $i=1,2,\ldots,n$ are invertible, which allows to derive a block $LU$-decomposition of $\mathcal{A}$ into invertible factors:
\[
  \mathcal{A} = 
  \begin{pmatrix} 
    I &  \\
    B_1 S_1^{-1} & I \\
     &  \ddots & \ddots \\
     & & (-1)^{n-2} B_{n-1} S_{n-1}^{-1}  & I
  \end{pmatrix}
  \begin{pmatrix}
    S_1 & B_1' \\
    & -S_2 & \ddots\\
    & & \ddots & B_{n-1}'\\
    & & & (-1)^{n-1} \, S_n
  \end{pmatrix} .
\]
So $\mathcal{A}$ is a bounded linear operator, which is invertible. Therefore, $\mathcal{A}$ is an isomorphism by the open mapping theorem.  
\end{proof}

With a slight abuse of notation we call $S_i$ Schur complements, although they are actually positive or negative Schur complements in the literal sense.

Under the assumptions made so far, we define the Schur complement preconditioner as the block diagonal linear operator $\mathcal{S}(\mathcal{A}) \colon \mathbf{X} \rightarrow \mathbf{X}'$, given by
\begin{equation}
\label{eq:SchurDefNN}
     \mathcal{S}(\mathcal{A}) =
    \begin{pmatrix} 
      S_1 &   &   & \\
          & S_2  &   &  \\
          &     & \ddots & \\
          &         &   & S_n
  \end{pmatrix}.
\end{equation}
If it is clear from the context which operator $\mathcal{A}$ is meant, we will omit the argument $\mathcal{A}$ and simply use $\mathcal{S}$ for the Schur complement preconditioner.
Since $\mathcal{S}$ is bounded, self-adjoint, and positive definite, it induces an inner product on $\mathbf{X}$ by
\begin{equation*}
\left(\mathbf{x},\mathbf{y}\right)_\mathcal{S} = \left\langle \mathcal{S}\mathbf{x},\mathbf{y}\right\rangle, 
\quad \text{for }\, \mathbf{x},\mathbf{y}\in \mathbf{X},
\end{equation*}
whose associated norm $\|\mathbf{x}\|_\mathcal{S} = \left(\mathbf{x},\mathbf{x}\right)_\mathcal{S}^{1/2}$ is equivalent to the canonical product norm in $\mathbf{X}$. Note that
\[
  \left(\mathbf{x},\mathbf{y}\right)_\mathcal{S} = \sum_{i=1}^n (x_i,y_i)_{S_i}
  \quad \text{with} \quad (x_i,y_i)_{S_i} = \langle S_i x_i,y_i\rangle.
\]
Here $x_i \in X_i$ and $y_i \in X_i$ denote the $i$-th component of $\mathbf{x} \in \mathbf{X}$ and $\mathbf{y} \in \mathbf{X}$, respectively.

From now on we assume that the spaces $\mathbf{X}$ and $\mathbf{X}'$ are equipped with the norm $\norm{.}_{\mathcal{S}}$ and the associated dual norm, respectively. The question whether (\ref{Axb}) is well-posed translates to the question whether $\mathcal{A} \colon \mathbf{X} \rightarrow \mathbf{X}'$ is an isomorphism.
The condition number $\kappa(\mathcal{A})$, given by 
\[
  \kappa\left(\mathcal{A}\right)
    = \norm{\mathcal{A}}_{\mathcal{L}\left(\mathbf{X},\mathbf{X}'\right)}
      \|\mathcal{A}^{-1}\|_{\mathcal{L}\left(\mathbf{X}',\mathbf{X}\right)},
\]
measures the sensitivity of the solution of (\ref{Axb}) with respect to data perturbations. Here $\mathcal{L}(X,Y)$ denotes the space of bounded linear operators from $X$ to $Y$, equipped with the standard operator norm. 

By the Riesz representation theorem the linear operator $\mathcal{S} \colon \mathbf{X} \rightarrow \mathbf{X}'$ is an isomorphism from $\mathbf{X}$ to $\mathbf{X}'$.
Therefore, $\mathcal{A}$ is an isomorphism if and only if $\mathcal{M}\colon \mathbf{X} \rightarrow \mathbf{X}$, given by 
\[
  \mathcal{M} = \mathcal{S}^{-1}\mathcal{A},
\]
is an isomorphism. In this context, the operator $\mathcal{S}$ can be seen as a preconditioner for $\mathcal{A}$ and $\mathcal{M}$ is the associated preconditioned operator. Moreover, it is easy to see that
\[
  \|\mathcal{A}\|_{\mathcal{L}\left(\mathbf{X},\mathbf{X}'\right)} = \|\mathcal{M}\|_{\mathcal{L}\left(\mathbf{X},\mathbf{X}\right)}
\]
and, in case of well-posedness,
\[
  \kappa\left(\mathcal{A}\right)
   = \kappa\left(\mathcal{M}\right)
   \quad \text{with} \quad
  \kappa\left(\mathcal{M}\right)
    = \norm{\mathcal{M}}_{\mathcal{L}\left(\mathbf{X},\mathbf{X}\right)}
      \|\mathcal{M}^{-1}\|_{\mathcal{L}\left(\mathbf{X},\mathbf{X}\right)}.
\]
The condition number $\kappa(\mathcal{M})$ is of significant influence on the convergence rate of preconditioned Krylov subspace methods for solving (\ref{Axb}).
We will now derive bounds for $\kappa(\mathcal{M})$, from which we will simultaneously learn about both the efficiency of the preconditioner $\mathcal{S}$ as well as the well-posedness of (\ref{Axb}) with respect to the norm $\norm{.}_{\mathcal{S}}$. See \cite{winther2011prec} for more on this topic of operator preconditioning.

We start the discussion by observing that 
  \begin{equation}
    \label{eq:CDef}
    \mathcal{M} = 
    \begin{pmatrix} 
    I & C^{*}_1  &  &\\
    C_1 & -\left(I-C_1C^{*}_1\right)  & \ddots &\\
       & \ddots& \ddots& C^{*}_{n-1}\\
        & & C_{n-1}   & (-1)^{n-1}\left(I-C_{n-1}C^{*}_{n-1}\right)
  \end{pmatrix},
\end{equation}
where
\begin{equation*}
C_{i}=S^{-1}_{i+1}B_i \quad\text{for } \, i = 1, 2,\ldots ,n-1.
\end{equation*}
For its Hilbert space adjoint $C^{*}_{i}$ with respect to the inner product $\left(\mathbf{x},\mathbf{y}\right)_\mathcal{S}$ we have the following representation:
\begin{equation*}
C^{*}_{i}=S^{-1}_iB_i',\quad\text{for } \, i = 1, 2,\ldots ,n-1.
\end{equation*}

In the next two theorems we will derive bounds for the norm of $\mathcal{M}$ and its inverse.

\begin{theorem}
\label{theo:CBound}
For the operator $\mathcal{M}$ the following estimate holds: 
\begin{equation*}
\norm{\mathcal{M}}_{\mathcal{L}\left(\mathbf{X},\mathbf{X}\right)} \leq 2\cos\left(\frac{\pi}{2n+1}\right).
\end{equation*}
\end{theorem}
\begin{proof}\unskip
First we note that the norm can be written in the following way: 
\begin{equation} \label{RQ}
\norm{\mathcal{M}}_{\mathcal{L}\left(\mathbf{X},\mathbf{X}\right)} 
= \sup_{0 \neq \mathbf{x}\in \mathbf{X}} \frac{\vert\left( \mathcal{M}\mathbf{x},\mathbf{x}\right)_{\mathcal{S}}\vert}{\left(\mathbf{x},\mathbf{x}\right)_{\mathcal{S}}}.
\end{equation}
We will now estimate the numerator $\left(\mathcal{M}\mathbf{x},\mathbf{x}\right)_{\mathcal{S}}$.
Let $\mathbf{x} \in X$ and let $x_i \in X_i$ be the $i$-th component of $\mathbf{x}$. Then it follows from (\ref{eq:CDef}) that
\begin{align*}
\left(\mathcal{M}\mathbf{x},\mathbf{x}\right)_{\mathcal{S}} = \norm{x_1}_{S_1}^2 + 2\sum^{n-1}_{i=1} \inner{C^{*}_{i} x_{i+1}}{x_i}_{S_i} 
 + \sum^{n-1}_{i=1} (-1)^{i} \inner{\left(I-C_{i}C^{*}_{i}\right)x_{i+1}}{x_{i+1}}_{S_{i+1}} .
  \end{align*}
By applying Cauchy's inequality and Young's inequality, we obtain for parameters $\epsilon_i > 0$ 
\begin{align*}
2\inner{C^{*}_{i} x_{i+1}}{x_i}_{S_i} \leq 2\norm{C^{*}_{i} x_{i+1}}_{S_i} \norm{x_i}_{S_i} 
&\leq \epsilon_i\inner{C^{*}_{i}x_{i+1}}{C^{*}_{i}x_{i+1}}_{S_i}+ \frac{1}{\epsilon_i} \norm{x_i}_{S_i}^2,\\
&= \epsilon_i\inner{C_{i}C^{*}_{i}x_{i+1}}{x_{i+1}}_{S_{i+1}} + \frac{1}{\epsilon_i} \norm{x_i}_{S_i}^2.
\end{align*}
Therefore,
\begin{align*}
\left(\mathcal{M}\mathbf{x},\mathbf{x}\right)_{\mathcal{S}} 
&\leq \norm{x_1}_{S_1}^2 + \sum^{n-1}_{i=1} \frac{1}{\epsilon_i}\norm{x_{i}}_{S_{i}}^2 
   +\sum^{n-1}_{i=1} \left( \epsilon_i - (-1)^i\right) \, \inner{C_{i}C^{*}_{i}x_{i+1}}{x_{i+1}}_{S_{i+1}} \\
 & \quad {} + \sum^{n-1}_{i=1} (-1)^{i}\norm{x_{i+1}}_{S_{i+1}}^2.
\end{align*}
Since $A_{i+1}$ is positive semi-definite, it follows that
\begin{align}
& \inner{C_iC^{*}_ix_{i+1}}{x_{i+1}}_{S_{i+1}} = \left\langle B_{i}S^{-1}_{i}B_{i}'x_{i+1}, x_{i+1}\right\rangle \nonumber\\
& \quad \le \left\langle A_{i+1}x_{i+1}, x_{i+1}\right\rangle +\left\langle B_{i}S^{-1}_{i}B_{i}'x_{i+1}, x_{i+1}\right\rangle 
= \left\langle S_{i+1}x_{i+1}, x_{i+1}\right\rangle = 
\norm{x_{i+1}}_{S_{i+1}}^2 \label{eq:proofCClowNN}
\end{align}
Now we make an essential assumption on the choice of the parameters $\epsilon_i$:
\begin{equation} \label{epsge1}
  \epsilon_i \ge 1 \quad \text{for} \ i = 1,\ldots,n-1.
\end{equation}
By using (\ref{epsge1}) and (\ref{eq:proofCClowNN}), the estimate for $\left(\mathcal{M}\mathbf{x},\mathbf{x}\right)_{\mathcal{S}}$ from above simplifies to 
\begin{align*}
\left(\mathcal{M}\mathbf{x},\mathbf{x}\right)_{\mathcal{S}} &\leq \norm{x_1}_{S_1}^2 + \sum^{n-1}_{i=1}\frac{1}{\epsilon_i} \norm{x_{i}}_{S_{i}}^2 +\sum^{n-1}_{i=1} \epsilon_i \norm{x_{i+1}}_{S_{i+1}}^2 = y^T D_u \, y,
\end{align*}
where
\begin{align*}
D_u = \begin{pmatrix} 
1+\frac{1}{\epsilon_1} & & & &\\
& \epsilon_1 +\frac{1}{\epsilon_2}& & &\\
& & \ddots & & \\
& & & \epsilon_{n-2} +\frac{1}{\epsilon_{n-1}} & \\
& & & & \epsilon_{n-1}
\end{pmatrix}
\quad \text{and} \quad
y = \begin{pmatrix} \|x_1\|_{S_1} \\ \|x_2\|_{S_2} \\ \vdots \\ \|x_n\|_{S_i} \end{pmatrix}.
\end{align*}
Next we choose $\epsilon_1, \ldots, \epsilon_{n-1}$ such that the diagonal elements in $D_u$ are all equal, that is
\[
  1 +\frac{1}{\epsilon_1} = \epsilon_{n-1}, \quad 
  \epsilon_1 +\frac{1}{\epsilon_2} = \epsilon_{n-1}, \quad \ldots,
  \epsilon_{n-2} +\frac{1}{\epsilon_{n-1}} = \epsilon_{n-1}.
\]
We can successively eliminate $\epsilon_1, \ldots, \epsilon_{n-2}$ from these equations and obtain
\begin{equation*}
  1   = \epsilon_{n-1} -\cfrac{1}{\epsilon_1} 
      = \epsilon_{n-1} -\cfrac{1}{\epsilon_{n-1} - \cfrac{1}{\epsilon_2} } 
      = \epsilon_{n-1} -\cfrac{1}{\epsilon_{n-1} - \cfrac{1}{\epsilon_{n-1} - \cfrac{1}{\epsilon_3}} }
      = \ldots \ ,
\end{equation*}
which eventually leads to the following equation for $\epsilon_{n-1}$:
\begin{equation} \label{eps}
  1 = \epsilon_{n-1} -\cfrac{1}{\epsilon_{n-1} - \cfrac{1}{\epsilon_{n-1} - \cfrac{1}{\epsilon_{n-1} - \ddots 
         \genfrac{}{}{0pt}{}{}{\cfrac{1}{\epsilon_{n-1}}} }}} \ .
\end{equation}
The right-hand side (\ref{eps}) is a continued fraction of depth $n-1$. It can easily be shown 
that this continued fraction is a rational function in $\epsilon_{n-1}$ of the form $P_n(\epsilon_{n-1})/P_{n-1}(\epsilon_{n-1})$, where $P_j(\epsilon)$ are polynomials of degree $j$, recursively given by
\begin{equation*}
  P_0 (\epsilon) = 1, \quad P_1(\epsilon) = \epsilon,
  \quad \text{and} \quad
  P_{i+1}(\epsilon) = \epsilon \, P_i(\epsilon)-P_{i-1}(\epsilon)
  \quad \text{for } i \ge 1.
\end{equation*}
Therefore, (\ref{eps}) becomes $1 = P_n(\epsilon_{n-1})/P_{n-1}(\epsilon_{n-1})$ or, equivalently,
\[
  \bar{P}_n(\epsilon_{n-1}) = 0 \quad \text{with} \quad
  \bar{P}_j(\epsilon) = P_j(\epsilon) - P_{j-1}(\epsilon).
\]
For the other parameters $\epsilon_1,\ldots,\epsilon_{n-2}$ it follows that
\[
  \epsilon_{n-i} = \frac{P_{i}(\epsilon_{n-1})}{P_{i-1}(\epsilon_{n-1}) }
  \quad \text{for} \ i = 2,\ldots,n-1.
\]
With this setting of the parameters the basic assumption (\ref{epsge1}) is equivalent to the following conditions:
\begin{equation} \label{ege1}
  \bar{P}_i(\epsilon_{n-1}) \ge 0
  \quad \text{and} \quad \epsilon_{n-1} \ge 1.
\end{equation}
To summarize, the parameter $\epsilon_{n-1}$ must be a root of $\bar{P}_n$ satisfying (\ref{ege1}). In the appendix it will be shown that
\[
  \epsilon_{n-1} = 2\cos\left(\frac{\pi}{2n+1}\right),
\]
which is the largest root of $\bar{P}_n$,
is an appropriate choice. Hence
\begin{align*}
\left(\mathcal{M}\mathbf{x},\mathbf{x}\right)_{\mathcal{S}} \leq y^T D_u \, y
= 2\cos\left(\frac{\pi}{2n+1}\right)\inner{\mathbf{x}}{\mathbf{x}}_{\mathcal{S}},
\end{align*}
and, therefore,
\[
  \frac{\left( \mathcal{M}\mathbf{x},\mathbf{x}\right)_{\mathcal{S}}}{\left(\mathbf{x},\mathbf{x}\right)_{\mathcal{S}}} 
  \le 2\cos\left(\frac{\pi}{2n+1}\right) .
\]
Following the same line of arguments a lower bound of $\left(\mathcal{M}\mathbf{x},\mathbf{x}\right)_{\mathcal{S}}$ can be derived,
\begin{align*}
\left(\mathcal{M}\mathbf{x},\mathbf{x}\right)_{\mathcal{S}} \geq y^T D_l \, y,
\end{align*}
where
\begin{align*}
D_l = \begin{pmatrix} 
1-\frac{1}{\epsilon_1} & & & &\\
& -\epsilon_1 -\frac{1}{\epsilon_2}& & &\\
& & \ddots & & \\
& & & -\epsilon_{n-2} -\frac{1}{\epsilon_{n-1}} & \\
& & & & -\epsilon_{n-1}
\end{pmatrix},
\end{align*}
with the same values for $\epsilon_i$ as before. From comparing $D_u$ and $D_l$ it follows that the diagonal elements of $D_l$ are equal to $-2\cos(\pi/(2n+1))$, except for the first element, which has the larger value $2-2\cos\left(\pi/(2n+1)\right)$.
This directly implies 
\begin{align*}
\left(\mathcal{M}\mathbf{x},\mathbf{x}\right)_{\mathcal{S}} \geq \inner{D_l\mathbf{x}}{\mathbf{x}}_{\mathcal{S}}\geq -2\cos\left(\frac{\pi}{2n+1}\right)\inner{\mathbf{x}}{\mathbf{x}}_{\mathcal{S}},
\end{align*}
and, therefore,
\[
  - 2\cos\left(\frac{\pi}{2n+1}\right) \le  \frac{\left( \mathcal{M}\mathbf{x},\mathbf{x}\right)_{\mathcal{S}}}{\left(\mathbf{x},\mathbf{x}\right)_{\mathcal{S}}}.
\]
To summarize we have shown that
\[
   \frac{\vert\left( \mathcal{M}\mathbf{x},\mathbf{x}\right)_{\mathcal{S}}\vert}{\left(\mathbf{x},\mathbf{x}\right)_{\mathcal{S}}} 
   \le 2\cos\left(\frac{\pi}{2n+1}\right),
\]
which completes the proof using (\ref{RQ}).
\end{proof}

For investigating the inverse operator $\mathcal{M}^{-1}$, notice first that $\mathcal{M} = \mathcal{M}_n$, where $\mathcal{M}_j$, $j=1,2,\ldots,n$ are recursively given by
\[
\mathcal{M}_1 = I, \quad
\mathcal{M}_{i+1} =     \begin{pmatrix} 
      \mathcal{M}_i & \mathbf{C}^{*}_i    \\
      \mathbf{C}_i & (-1)^i \, (I - C_i C_i^*)  
    \end{pmatrix}
  \quad \text{with} \ 
   \mathbf{C}_i = \begin{pmatrix} 0 & \ldots &0 &C_i \end{pmatrix}
  \quad  \text{for} \ i \ge 1.
\]
Under the assumptions of Lemma \ref{lem:one} one can show by elementary calculations that $\mathcal{M}_j^{-1}$, $j=1,2,\ldots,n$ exist and satisfy the  following recurrence relation:
\begin{equation}
\label{eq:invRecfirstKind1}
\mathcal{M}^{-1}_1 = I, \quad
\mathcal{M}^{-1}_{i+1} = 
  \begin{pmatrix} 
      \mathcal{M}^{-1}_i & 0 \\
      0 & 0
  \end{pmatrix} + (-1)^i
  \begin{pmatrix} 
      \mathcal{M}^{-1}_i\mathbf{C}^{*}_i\mathbf{C}_i \mathcal{M}^{-1}_i & 
      - \mathcal{M}^{-1}_i\mathbf{C}^{*}_i    \\
      - \mathbf{C}_i\mathcal{M}^{-1}_i & I
  \end{pmatrix}
  \quad  \text{for} \ i \ge 1.
\end{equation}

\begin{theorem}
\label{theo:CInvBound}
The operator $\mathcal{M}$ is invertible and we have
\begin{equation*}
  \norm{\mathcal{M}^{-1}}_{\mathcal{L}\left(\mathbf{X},\mathbf{X}\right)} 
    \leq \frac{1}{2\sin\left(\frac{\pi}{2\left(2n+1\right)}\right)} .
\end{equation*}
\end{theorem}
\begin{proof}\unskip
Let $\mathbf{x} \in \mathbf{X} = X_1\times \ldots \times X_n$ with components $x_i  \in X_i$. The restriction of $\mathbf{x}$ to its first $j$ components is denoted by $\mathbf{x}_j \in X_1\times \ldots \times X_j$.

From (\ref{eq:invRecfirstKind1}) we obtain
\begin{align*}
 (\mathcal{M}_{i+1}^{-1} \mathbf{x}_{i+1},\mathbf{x}_{i+1})_{\mathcal{S}_{i+1}}
   & = (\mathcal{M}_i^{-1} \mathbf{x}_i,\mathbf{x}_i)_{\mathcal{S}_i} 
       + (-1)^i \, \|\mathbf{C}_i\mathcal{M}^{-1}_i \mathbf{x}_i - x_{i+1}\|_{S_{i+1}}^2,
\end{align*}
which implies that
\begin{equation} \label{oneside}
   (\mathcal{M}_{i+1}^{-1} \mathbf{x}_{i+1},\mathbf{x}_{i+1})_{\mathcal{S}_{i+1}}
     \begin{cases}
       \le  (\mathcal{M}_i^{-1} \mathbf{x}_i,\mathbf{x}_i)_{\mathcal{S}_i} & \text{for odd} \ i , \\ 
       \ge  (\mathcal{M}_i^{-1} \mathbf{x}_i,\mathbf{x}_i)_{\mathcal{S}_i} & \text{for even} \ i .
     \end{cases}
\end{equation}
For estimates in the opposite direction observe that
(\ref{eq:invRecfirstKind1}) also implies that
\begin{equation} \label{oppdir}
\begin{aligned}
 (\mathcal{M}_{i+1}^{-1} \mathbf{x}_{i+1},\mathbf{x}_{i+1})_{\mathcal{S}_{i+1}}
   & = ([\mathcal{M}_i^{-1} + (-1)^i \, \mathcal{M}^{-1}_i\mathbf{C}^{*}_i\mathbf{C}_i \mathcal{M}^{-1}_i] \,  \mathbf{x}_i,\mathbf{x}_i)_{\mathcal{S}_i} \\
   & \quad {} + (-1)^i \, \left[ \norm{x_{i+1}}_{S_{i+1}}^2 -  2 \, (\mathbf{C}_i\mathcal{M}^{-1}_i \mathbf{x}_i,x_{i+1})_{S_{i+1}} \right].
\end{aligned}
\end{equation}
For the first term on the right-hand side of (\ref{oppdir}) observe that
\begin{equation*}
  \mathcal{M}^{-1}_i 
   + (-1)^i \, \mathcal{M}^{-1}_i \mathbf{C}^{*}_i \mathbf{C}_i \mathcal{M}^{-1}_i 
    = \begin{pmatrix}
        \mathcal{M}_{i-1}^{-1} & 0 \\ 0 & 0
      \end{pmatrix}  
        + (-1)^{i-1} \, \mathcal{P}_{i}
\end{equation*}
with
\[
  \mathcal{P}_i = \begin{pmatrix}
      \mathcal{M}_{i-1}^{-1} \mathbf{C}_{i-1}^* (I -C_i^* \, C_i) \mathbf{C}_{i-1} \mathcal{M}_{i-1}^{-1}  &
      -\mathcal{M}_{i-1}^{-1} \mathbf{C}_{i-1}^* (I -C_i^* \, C_i)\\
      - (I -C_i^* \, C_i) \mathbf{C}_{i-1} \mathcal{M}_{i-1}^{-1} &  (I -C_i^* \, C_i)
    \end{pmatrix},
\]
which easily follows by using (\ref{eq:invRecfirstKind1}) with $i$ replaced by $i-1$. The operator $\mathcal{P}_i$ is positive semi-definite, since
\begin{equation*}
  ( \mathcal{P}_i \mathbf{x}_i, \mathbf{x}_i )_{\mathcal{S}_i} 
   = ([I -C_i^* \, C_i](\mathbf{C}_{i-1} \mathcal{M}_{i-1}^{-1} \mathbf{x}_{i-1} - x_i),
                        \mathbf{C}_{i-1} \mathcal{M}_{i-1}^{-1} \mathbf{x}_{i-1} - x_i)_{S_i} \ge 0.
\end{equation*}
Therefore
\[
  ([\mathcal{M}^{-1}_i 
   + (-1)^i \, \mathcal{M}^{-1}_i \mathbf{C}^{*}_i \mathbf{C}_i \mathcal{M}^{-1}_i] \mathbf{x}_i,\mathbf{x}_i)_{\mathcal{S}_i}
   \begin{cases} 
     \ge (\mathcal{M}_{i-1}^{-1} \mathbf{x}_{i-1},\mathbf{x}_{i-1})_{\mathcal{S}_{i-1}} & \text{for odd} \ i, \\
     \le (\mathcal{M}_{i-1}^{-1} \mathbf{x}_{i-1},\mathbf{x}_{i-1})_{\mathcal{S}_{i-1}} & \text{for even} \ i.
     \end{cases}
\]
Then it follows from (\ref{oppdir}) that for odd $i$
\[
  (\mathcal{M}_{i+1}^{-1} \mathbf{x}_{i+1} , \mathbf{x}_{i+1})_{\mathcal{S}_{i+1}}
   \ge (\mathcal{M}_{i-1}^{-1} \mathbf{x}_{i-1},\mathbf{x}_{i-1})_{\mathcal{S}_{i-1}}
     + 2 (\mathbf{C}_i \mathcal{M}_i^{-1} \mathbf{x}_i,x_{i+1})_{S_{i+1}} - \|x_{i+1}\|_{S_{i+1}}^2
\]
and for even $i$
\[
  (\mathcal{M}_{i+1}^{-1} \mathbf{x}_{i+1} , \mathbf{x}_{i+1})_{\mathcal{S}_{i+1}}
   \le (\mathcal{M}_{i-1}^{-1} \mathbf{x}_{i-1},\mathbf{x}_{i-1})_{\mathcal{S}_{i-1}}
     - 2 (\mathbf{C}_i \mathcal{M}_i^{-1} \mathbf{x}_i,x_{i+1})_{S_{i+1}} + \|x_{i+1}\|_{S_{i+1}}^2 .
\]
In order to estimate $(\mathbf{C}_i \mathcal{M}_i^{-1} \mathbf{x}_i,x_{i+1})_{S_{i+1}}$  observe that
\[
  \mathbf{C}_i \mathcal{M}_i^{-1} 
     = -(-1)^i \, C_i \, 
       \begin{pmatrix}
        -\mathbf{C}_{i-1} \mathcal{M}_{i-1}^{-1} &  I
      \end{pmatrix} ,
\]
which is obtained from (\ref{eq:invRecfirstKind1}) with $i$ replaced by $i-1$ by multiplying with $\mathbf{C}_i$ from the left. Therefore,
\begin{align*}
  \|\mathbf{C}_i \mathcal{M}_i^{-1} \mathbf{x}_i\|_{S_{i+1}}
     & \le \| C_i \mathbf{C}_{i-1} \mathcal{M}_{i-1}^{-1} \mathbf{x}_{i-1}\|_{S_{i+1}} + \|C_i x_i\|_{S_{i+1}} \\
     & \le \|\mathbf{C}_{i-1} \mathcal{M}_{i-1}^{-1} \mathbf{x}_{i-1}\|_{S_i} + \|x_i\|_{S_i},
\end{align*}
which recursively applied eventually leads to
\[
  \|\mathbf{C}_i \mathcal{M}_i^{-1} \mathbf{x}_i\|_{S_{i+1}} 
    \le \sum_{j=1}^i \|x_j\|_{S_j} .
\]
Hence
\[
  \left|(\mathbf{C}_i \mathcal{M}_i^{-1} \mathbf{x}_i,x_{i+1})_{S_{i+1}}\right|
  \le  \sum_{j=1}^i \|x_j\|_{S_j} \, \|x_{i+1}\|_{S_{i+1}}.
\]
Using this estimate, we obtain for odd $i$, 
\begin{align*}
  (\mathcal{M}_{i+1}^{-1} \mathbf{x}_{i+1} , \mathbf{x}_{i+1})_{\mathcal{S}_{i+1}}
   & \ge (\mathcal{M}_{i-1}^{-1} \mathbf{x}_{i-1},\mathbf{x}_{i-1})_{\mathcal{S}_{i-1}}
     - 2 \sum_{j=1}^i \|x_j\|_{S_j} \, \|x_{i+1}\|_{S_{i+1}} - \|x_{i+1}\|_{S_{i+1}}^2 \\
   &  =  (\mathcal{M}_{i-1}^{-1} \mathbf{x}_{i-1},\mathbf{x}_{i-1})_{\mathcal{S}_{i-1}}
     + y_{i+1}^T L_{i+1} \, y_{+1}.
\end{align*}
where $y_j= (\norm{x_1}_{S_1},\norm{x_2}_{S_2},\ldots,\norm{x_j}_{S_j})^T$ and $L_ {i+1}$ is the $(i+1)\times(i+1)$ matrix whose only nonzero entries are $-1$ in the last row and last column. 

Applying this estimate recursively, eventually leads to 
\[
  (\mathcal{M}_j^{-1} \mathbf{x}_j , \mathbf{x}_j)_{\mathcal{S}_j}
   \ge y_j^T Q_j \, y_j ,
\]
where $Q_j$, $j=2,4,6,\ldots$ are given by the recurrence relation
\[
  Q_2 = \begin{pmatrix}  0 & -1 \\ -1 & -1\end{pmatrix}, \quad 
  Q_{i+2} = 
  \left(\begin{array}{@{}c|cc@{}} 
     Q_{i} & \begin{matrix} 0 \\ \vdots \end{matrix} & \begin{matrix} -1 \\ \vdots \end{matrix} \\ \hline
     \begin{matrix} 0 & \cdots \end{matrix} & 0 & -1\\ 
     \begin{matrix} -1 & \cdots \end{matrix} & -1 & -1
     \end{array}\right)
  \quad \text{for} \ i = 2,4,\ldots .
\]
Therefore,
\[
  (\mathcal{M}_j^{-1} \mathbf{x}_j , \mathbf{x}_j)_{\mathcal{S}_j}
   \ge - \|Q_j\| \, \inner{\mathbf{x}_j}{\mathbf{x}_j}_{\mathcal{S}_j}
   \quad \text{for even} \ j,
\]
where $\|Q_j\|$ denotes the spectral norm of $Q_j$.
It follows analogously that
\[
  (\mathcal{M}_j^{-1} \mathbf{x}_j , \mathbf{x}_j)_{\mathcal{S}_j}
   \le \|Q_j\| \, \inner{\mathbf{x}_j}{\mathbf{x}_j}_{\mathcal{S}_j}
   \quad \text{for odd} \ j,
\]
where $Q_j$, $j=1,3,5,\ldots$ are given by the recurrence relation
\[
  Q_1 = 1, \quad 
  Q_{i+2} = 
  \left(\begin{array}{@{}c|cc@{}} 
     Q_{i} & \begin{matrix} 0 \\ \vdots \end{matrix} & \begin{matrix} 1 \\ \vdots \end{matrix} \\ \hline
     \begin{matrix} 0 & \cdots \end{matrix} & 0 & 1\\ 
     \begin{matrix} 1 & \cdots \end{matrix} & 1 & 1
     \end{array}\right)
  \quad \text{for} \ i = 1,3,\ldots
\]
Together with (\ref{oneside}) it follows for odd $i$ that
\begin{equation*}
   - \|Q_{i+1}\| \, \inner{\mathbf{x}_{i+1}}{\mathbf{x}_{i+1}}_{\mathcal{S}_{i+1}} 
     \le (\mathcal{M}_{i+1}^{-1} \mathbf{x}_{i+1},\mathbf{x}_{i+1})_{\mathcal{S}_{i+1}}
     \le (\mathcal{M}_i^{-1} \mathbf{x}_i,\mathbf{x}_i)_{\mathcal{S}_i}
     \le \|Q_i\| \, \inner{\mathbf{x}_i}{\mathbf{x}_i}_{\mathcal{S}_i},
\end{equation*}
and for even $i$ that
\begin{equation*}
   - \|Q_i\| \, \inner{\mathbf{x}_i}{\mathbf{x}_i}_{\mathcal{S}_i} 
     \le (\mathcal{M}_i^{-1} \mathbf{x}_i,\mathbf{x}_i)_{\mathcal{S}_i}
     \le (\mathcal{M}_{i+1}^{-1} \mathbf{x}_{i+1},\mathbf{x}_{i+1})_{\mathcal{S}_{i+1}}
     \le \|Q_{i+1}\| \, \inner{\mathbf{x}_{i+1}}{\mathbf{x}_{i+1}}_{\mathcal{S}_{i+1}}.
\end{equation*}
So in both cases we obtain
\[
   \frac{|(\mathcal{M}_{i+1}^{-1} \mathbf{x}_{i+1},\mathbf{x}_{i+1})_{\mathcal{S}_{i+1}}|}%
   {\inner{\mathbf{x}_{i+1}}{\mathbf{x}_{i+1}}_{\mathcal{S}_{i+1}}} \le \max(\|Q_i\|,\|Q_{i+1}\|)
   \quad \foralls \ \mathbf{x}_{i+1} \neq 0.
\]
Since
\[
  \|Q_j\| = \frac{1}{2\sin\left(\frac{\pi}{2\left(2 j+1\right)}\right)},
\]
see the appendix, the proof is completed.
\end{proof}

A direct consequence of Theorems \ref{theo:CBound} and \ref{theo:CInvBound} is the following corollary.

\begin{corollary}
\label{corollary:condNumberNN}
Under the assumptions of Lemma \ref{lem:one} the block operators $\mathcal{A}$ and $\mathcal{M}$, given by (\ref{eq:ADefNN}) and (\ref{eq:CDef}), are  isomorphisms from $\mathbf{X}$ to $\mathbf{X}'$ and from $\mathbf{X}$ to $\mathbf{X}$, respectively. Moreover, the following condition number estimate holds
\begin{equation*}
  \kappa\left(\mathcal{A}\right)= \kappa\left(\mathcal{M}\right)
    \leq \frac{\cos\left(\frac{\pi}{2n+1}\right)}{\sin\left(\frac{\pi}{2\left(2n+1\right)}\right)} .
\end{equation*}
\end{corollary}

\begin{remark}
For $n=2$ we have
\[
  2 \sin \left( \frac{\pi}{10} \right) = \frac{1}{2}(\sqrt{5}-1) 
  \quad \text{and} \quad
  2 \cos \left( \frac{\pi}{5} \right) = \frac{1}{2}(\sqrt{5}+1),
\]
and, therefore,
\begin{equation*}
  \kappa\left(\mathcal{M}\right) \le \frac{\sqrt{5}+1}{\sqrt{5}-1} =  \frac{3+\sqrt{5}}{2}.
\end{equation*}
This result is well known for finite dimensional spaces, see \cite{kuznetsov95},  \cite{murphy2000note}.
\end{remark}

In \cite{kuznetsov95},  \cite{murphy2000note} it was also shown for the case  $n=2$ and $A_2 = 0$ that $\mathcal{M}$ has only 3 eigenvalues:
\[
  \left\{-\frac{1}{2}(\sqrt{5}-1), 1,  \frac{1}{2}(\sqrt{5}+1) \right\}.
\]
This result can also be extended for $n \ge 2$ and for general Hilbert spaces.

\begin{theorem}
\label{theo:exactEigV}
If the assumptions of Lemma \ref{lem:one} hold and if, additionally, $A_i=0$, for $i=2,\ldots,n$, then 
the set $\sigma_p(\mathcal{M})$ of all eigenvalues of $\mathcal{M}$, is given by
\begin{equation*}
 \sigma_p(\mathcal{M}) = \left\lbrace  2\cos\left(\frac{2i-1}{2j+1}\pi\right) \colon j = 1,\ldots,n, \ i = 1,\ldots, j \right\rbrace.
\end{equation*}
Moreover,
\[
  \|\mathcal{M}\|_{\mathcal{L}(\mathbf{X},\mathbf{X})} = \max \{|\lambda| : \lambda \in  \sigma_p(\mathcal{M}) \} = 2\cos\left(\frac{\pi}{2n+1}\right)
\]
and
\[
  \|\mathcal{M}^{-1}\|_{\mathcal{L}(\mathbf{X},\mathbf{X})} = \max \left\{\frac{1}{|\lambda|} : \lambda \in  \sigma_p(\mathcal{M}) \right\}
   = \frac{1}{2\sin\left(\frac{\pi}{2\left(2n+1\right)}\right)}.
\]
So, equality is attained in the estimates of Theorems \ref{theo:CBound}, \ref{theo:CInvBound}, and Corollary \ref{corollary:condNumberNN}. All estimates are sharp.
\end{theorem}
\begin{proof}\unskip
Since $A_i=0$, for $i=2,\ldots,n$ it follows that $C_iC^{*}_i = I$ and the block operator $\mathcal{M}$ simplifies to
  \begin{equation*}
    \mathcal{M} = 
    \begin{pmatrix} 
    I & C^{*}_1  &  &\\
    C_1 & 0  & \ddots &\\
       & \ddots& \ddots& C^{*}_{n-1}\\
        & & C_{n-1}   & 0
  \end{pmatrix}.
\end{equation*}
The eigenvalue problem $\mathcal{M} \mathbf{x} = \lambda \, \mathbf{x}$ reads in details:
\begin{align*}
x_1 +C^{*}_{1}x_2 &= \lambda x_1,\\
C_1x_1 +C^{*}_{2}x_3 &= \lambda x_2,\\
&\vdots\\
C_{n-2}x_{n-2} +C^{*}_{n-1}x_{n} &= \lambda x_{n-1},\\
C_{n-1}x_{n-1} &= \lambda x_{n}.
\end{align*}
From the first equation
\begin{equation*} 
  C^{*}_{1} x_2 = \bar{P}_1(\lambda) \, x_1 ,
  \quad \text{where} \quad \bar{P}_1(\lambda) = \lambda - 1,
\end{equation*}
we conclude that the root  $\lambda_{11} = 1$ of $\bar{P}_1(\lambda)$ is an eigenvalue by setting $x_2 = x_3 = \ldots = x_n = 0$. If $\lambda \neq \lambda_{11}$, then we obtain from the second equation by eliminating $x_1$:
\[
  C^{*}_{2}x_3 
    = C_1 x_1 - \lambda \, x_2 
    = \frac{1}{\bar{P}_1(\lambda)} \,  C_1 C_1^* x_2 - \lambda \, x_2 
    = \frac{1}{\bar{P}_1(\lambda)} \,  x_2 - \lambda \, x_2 
    = R_2(\lambda) \, x_2,
\]
where
\[ 
  R_2(\lambda) = \lambda -\frac{1}{\bar{P}_1(\lambda)} = \frac{\bar{P}_2(\lambda)}{\bar{P}_1(\lambda)}
  \quad   \text{with} \quad \bar{P}_2(\lambda) = \lambda \, \bar{P}_1(\lambda) -1.
\]
We conclude that the two roots $\lambda_{21}$ and $\lambda_{22}$ of the polynomial $\bar{P}_2(\lambda)$ of degree 2 are eigenvalues by setting $x_3 = \ldots = x_n = 0$.
Repeating this procedure gives
\begin{align*} 
  C^{*}_{j}x_{j+1} = R_j(\lambda) \,  x_j,
  \text{ for } j=2,\ldots,n-1,
  \text{ and }  
  0 = R_n(\lambda) \, x_n
  \text{ with } R_j(\lambda) = \frac{\bar{P}_j(\lambda)}{\bar{P}_{j-1}(\lambda)},
\end{align*}
where the polynomials $\bar{P}_j(\lambda)$ are recursively given by
\[
  \bar{P}_0(\lambda) = 1,\quad
  \bar{P}_1(\lambda) = \lambda - 1, \quad 
  \bar{P}_{i+1}(\lambda) = \lambda \, \bar{P}_i(\lambda) - \bar{P}_{i-1}(\lambda)
  \quad \text{for} \, i \ge 1.
\]
So the eigenvalues of $\mathcal{M}$ are the roots of the polynomials $\bar{P}_1\left(\lambda \right),\bar{P}_2(\lambda),\ldots,\bar{P}_n\left(\lambda \right)$.
For the roots of $\bar{P}_j\left(\lambda \right)$ we obtain
\begin{equation*}
 \lambda  = 2 \,\cos\left(\frac{2i-1}{2j+1}\pi\right), \quad \text{for } i= 1,\ldots,j,
\end{equation*}
see Lemma \ref{lem:CheByRootsandRealProblem}. It is easy to see that
\[
  \max \{|\lambda| : \lambda \in  \sigma_p(\mathcal{M}) \} = 2\cos\left(\frac{\pi}{2n+1}\right).
\]
Therefore, with Theorem \ref{theo:CBound} it follows that
\[
  2\cos\left(\frac{\pi}{2n+1}\right) \ge \|\mathcal{M}\|_{\mathcal{L}(\mathbf{X},\mathbf{X})}
    \ge \max \{|\lambda| : \lambda \in  \sigma_p(\mathcal{M}) \} = 2\cos\left(\frac{\pi}{2n+1}\right),
\]
which implies equality. An analogous argument applies to $\mathcal{M}^{-1}$.
\end{proof}

%% file: OCPoisson.tex
\section{Application: Optimization problems in function space}
\label{sec:OCPoissonProblem}

In this section we apply the theory from Section \ref{sec:Theory} to optimization problems in function spaces with an elliptic partial differential equation as constraint. First we present a standard model problem. Then we look at two more challenging variations of the model problem in Subsections \ref{subsec:limitedCntl} and \ref{subsec:limitedObs}.

\subsection{Distributed observation and distributed control}
\label{subsec:distrObsdistrCntl}

We start with the following model problem: 
Find $u \in U$ and $f \in F = L^2(\Omega)$ which minimizes the objective functional
  \begin{align*}
    \frac12 \normTwo{u-d}_{\lTwo} + \frac{\alpha}{2}\normTwo{f}_{\lTwo}  
  \end{align*}
subject to the constraint
  \begin{align*}
    -\Delta u + f &= 0\quad \text{in} \quad \Omega,\\
     u &= 0 \quad \text{on} \quad \partial\Omega,
  \end{align*}
where $\Omega$ is a bounded open subset of $\mathbb{R}^d$ with Lipschitz boundary $\partial \Omega$, $\Delta$ is the Laplacian operator, $d\in \lTwo$, $\alpha > 0$ are given data, and $U \subset \lTwo$ is a Hilbert space. Here
$L^2(\Omega)$ denotes the standard Lebesgue space of square-integrable functions on $\Omega$ with inner product $\inner{\cdot}{\cdot}_{L^2(\Omega)}$ and norm $\norm{\cdot}_{L^2(\Omega)}$.

This problem can be seen either as an inverse problem for identifying $f$ from the data $d$, or as an optimal control problem with state $u$, control $f$, and the desired state $d$. In the first case the parameter $\alpha$ is a regularization parameter, in the second case a cost parameter.
Throughout this paper, we adopt the terminology of an optimal control problem and call $U$ the state space and $F$ the control space.

We discuss now the construction of preconditioners for the associated optimality system such that the condition number of the preconditioned system is bounded independently of $\alpha$.
We will call such preconditioners $\alpha$-robust. This is of particular interest in the context of inverse problems, where $\alpha$ is typically small, in which case the unpreconditioned operator becomes severely ill-posed. 

The problem is not yet fully specified. We need a variational formulation for the constraint, which will eventually lead to the definition of the state space $U$.

The most natural way is to use the standard weak formulation with $U =H_0^1(\Omega)$:
  \begin{align*}
     \innerL{\nabla u}{\nabla w} +\innerL{f}{w} &= 0\quad \foralls w\in W = H_0^1(\Omega),
  \end{align*}
where $\nabla$ denotes the gradient of a scalar function.
Here we use $H^m(\Omega)$ ($H_0^m(\Omega)$) to denote the standard Sobolev spaces of functions on $\Omega$ (with vanishing trace) with associated norm $\norm{\cdot}_{H^m(\Omega)}$ ($|\cdot|_{H^m(\Omega)}$) .
This problem is well studied, see, for example, \cite{troltzsch2010optimal}. 

Other options for the state equation are the very weak form in the following sense: $U = L^2(\Omega)$ and
  \begin{align*}
    - \innerL{u}{\Delta w}  +\innerL{f}{w} &= 0\quad \foralls w\in W = H^2(\Omega) \cap H_0^1(\Omega),
  \end{align*}
and the strong form with $U = H^2(\Omega) \cap H_0^1(\Omega)$ and
  \begin{align*}
     -\innerL{\Delta u}{ w} +\innerL{f}{w} &= 0\quad \foralls w\in W = L^2(\Omega).
  \end{align*}
In each of these different formulations of the optimization problem only two bilinear forms are involved: the $L^2$-inner product (twice in the objective functional and once in the state equation as the second term) and the bilinear form representing the negative Laplacian. In operator notation we use $M \colon L^2(\Omega) \rightarrow (L^2(\Omega))'$ for representing the $L^2$-inner product and $K \colon U \rightarrow U'$ for representing the bilinear form associated to the negative Laplacian:
\[
  \langle M y,z \rangle = \inner{y}{z}_{L^2(\Omega)}, \quad
  \langle K u,w \rangle 
    = \begin{cases}
       \innerL{\nabla u}{\nabla v}, \\
       -\innerL{u}{\Delta v}, \\
       -\innerL{\Delta u}{v},
      \end{cases}
\]
depending on the choice of $U$.
With these notations the state equation reads:
\[
  \dualprod{K u}{w} + \dualprod{M f}{w} = 0 \quad \foralls w\in 
   W = \begin{cases}
         H_0^1(\Omega), \\
         H^2(\Omega) \cap H_0^1(\Omega), \\
         L^2(\Omega).
       \end{cases}
\]

For discretizing the problem we use the optimize--then--discretize approach. Therefore we start by introducing
the associated Lagrangian functional, which reads
\begin{equation*}
\mathcal{L}\left(u,f,w\right) = \frac12 \normTwo{u-d}_{\lTwo} + \frac{\alpha}{2}\normTwo{f}_{\lTwo} +\dualprod{Ku}{w} +\dualprod{Mf}{w},
\end{equation*}
with $u\in U$, $f\in F$ and the Lagrangian multiplier $w\in W$.

From the first order necessary optimality conditions
\begin{equation*}
\pder{\mathcal{L}}{u}\left(u,f,w\right)=0, \quad 
\pder{\mathcal{L}}{f}\left(u,f,w\right)=0, \quad 
\pder{\mathcal{L}}{w}\left(u,f,w\right)=0,
\end{equation*}
which are also sufficient here, we obtain the optimality system, which leads to the following problem:

\begin{problem}
\label{prob:simpleAOCKKTstadard}
Find $\left(u,f,w\right)\in U\times F \times W$ such that 
\begin{equation*}
  \mathcal{A}_\alpha
  \begin{pmatrix} 
    u \\
    f \\
    w   
  \end{pmatrix} = 
  \begin{pmatrix} 
    {M} d\\
    0 \\
    0    
  \end{pmatrix}
  \quad \text{with} \quad
  \mathcal{A}_\alpha
   = \begin{pmatrix} 
    M &    0     & K' \\
    0 & \alpha M & M \\
    K &    M    & 0 
  \end{pmatrix} .
\end{equation*}
\end{problem}
Strictly speaking, the four operators $M$ appearing in $\mathcal{A}_\alpha$ are restrictions of the original operator $M$ introduced above on the corresponding spaces  $U$, $F$, and $W$.

The block operator in Problem \ref{prob:simpleAOCKKTstadard} is of the form (\ref{eq:ADefNN}) for $n = 2$ with
\begin{equation} 
\label{A12B1}
A_1 = 
\begin{pmatrix} 
    M  & 0  \\
    0  & \alpha M   
\end{pmatrix},\quad A_2 = 0
\quad\text{and}\quad B_1 = \begin{pmatrix} 
    K  & M  
\end{pmatrix}.
\end{equation}
We now analyze the three possible choices of $U$, which were considered above: 
\begin{enumerate}
\item
We start with $U =H_0^1(\Omega)$, where $\dualprod{Ku}{w} = \innerL{\nabla u}{\nabla w}$ and $W=H_0^1(\Omega)$. In this case it is obvious that $A_1$ is not positive definite on $X_1 = U \times F = H_0^1(\Omega) \times L^2(\Omega)$. So the results of Section \ref{sec:Theory} do not apply.
However, there exist other preconditioners which are $\alpha$-robust for this choice of $U =H_0^1(\Omega)$, see \cite{schoberl2007symmetric}.
\item
Next we examine $U = L^2(\Omega)$, where $\dualprod{Ku}{v}= -\innerL{u}{\Delta v}$ and $W = H^2(\Omega) \cap H_0^1(\Omega)$. For this choice, it is easy to see that  $S_1$ is positive definite, $S_2$ is well-defined with
\begin{equation} 
\label{case2}
S_1 = \begin{pmatrix} 
    M  & 0  \\
    0  & \alpha M   
\end{pmatrix}\quad \text{and} \quad
S_2 =  \frac{1}{\alpha} \, M+K M^{-1}K'.
\end{equation}
In order to apply the results of Section \ref{sec:Theory} we are left with showing that $S_2$ is positive definite. 
First observe that we have the following alternative representation of $S_2$:
\begin{lemma}
\[
  S_2 = \frac{1}{\alpha} \, M + B,
\]
where the (biharmonic) operator $B$ is given by 
\begin{equation} \label{biharm}
  \dualprod{B y}{z} = \innerL{\Delta y}{\Delta z}, \quad \foralls y, z \in H^2(\Omega) \cap H_0^1(\Omega).
\end{equation}
\end{lemma}
\begin{proof}
For $w \in W = H^2(\Omega) \cap H_0^1(\Omega)$, We have 
\begin{align*}
\dualprod{K M^{-1} K' w}{w}
  & = \sup_{0 \neq v \in U}\frac{\dualprod{K' w}{v}^{2}}{\dualprod{M v}{v}} 
    = \sup_{0 \neq v \in U}\frac{\inner{v}{-\Delta w}^{2}_{\lTwoB}}{\inner{v}{v}_{L^2\left(\Omega\right)}}
    = \norm{\Delta w}_{\lTwo}^2,
\end{align*}
from which it follows that $K M^{-1} K' = B$, since both operators are self-adjoint. 
\end{proof}

The second ingredient for showing the positive definiteness of  $S_2$ is the following result, see \cite{grisvard2011elliptic} for a proof:

\begin{theorem} $\,$
\label{theo:convexGrisvard}
Assume that $\Omega$ is a bounded open subset in $\mathbb{R}^2$ ($\mathbb{R}^3$) with a polygonal (polyhedral) boundary $\partial \Omega$. Then 
\begin{equation*}
\norm{v}_{H^2\left(\Omega\right)} \leq C_r \norm{\Delta v}_{\lTwo}, \quad  \foralls v \in H^2\left(\Omega\right)\cap H^1_0\left(\Omega\right),
\end{equation*}
for some constant $C_r$.
\end{theorem}

From this a priori estimate the required property of $S_2$ follows:

\begin{lemma}
\label{lemma:S3SPD}
Under the assumptions of Theorem \ref{theo:convexGrisvard}, the operator $ S_2: W \rightarrow W'$, given by (\ref{case2}) is bounded, self-adjoint and positive definite, where 
$W = H^2\left(\Omega\right)\cap H^1_0\left(\Omega\right)$.
\end{lemma}
\begin{proof}\unskip
It is obvious that $S_2$ is bounded and self-adjoint. Moreover,
it follows from Theorem \ref{theo:convexGrisvard} that $B$ is positive definite. Since $S_2 \ge B$, $S_2$ is positive definite, too.
\end{proof}
As a direct consequence from Corollary \ref{corollary:condNumberNN}, Lemma \ref{lemma:S3SPD}, and the results of Section \ref{sec:Theory} we have

\begin{corollary}
\label{corollary:condNumberOCObservation}
Under the assumptions of Theorem \ref{theo:convexGrisvard}, the operator $\mathcal{A}_\alpha$ in Problem \ref{prob:simpleAOCKKTstadard} is an isomorphism from $\mathbf{X} = \lTwo\times \lTwo \times \left(H^2\left(\Omega\right)\cap H^1_0\left(\Omega\right)\right)$ to its dual space with respect to the norm in $\mathbf{X}$ given by
\begin{equation*}
\normTwo{\left(u,f,w\right)} = \normTwo{u}_U + \normTwo{f}_F + \normTwo{w}_W
\end{equation*}
with
\begin{equation*}
\normTwo{u}_{U} =  \normTwo{u}_{\lTwo},\quad 
\normTwo{f}_{F} =  \alpha \normTwo{f}_{\lTwo},\quad 
\normTwo{w}_{W} =  \frac{1}{\alpha}\normTwo{w}_{\lTwo} + \normTwo{\Delta w}_{\lTwo}. 
\end{equation*}
Furthermore, the following condition number estimate holds:
\begin{equation*}
\kappa\left(\mathcal{S}^{-1}_\alpha\mathcal{A}_\alpha\right) \leq \frac{\cos\left(\frac{\pi}{5}\right)}{\sin\left(\frac{\pi}{10}\right)}\approx 2.62
\quad \text{with} \quad 
\mathcal{S}_\alpha 
  = \begin{pmatrix} 
      M & & \\
      & \alpha \, M & \\
      & & \frac{1}{\alpha} M + B  
    \end{pmatrix}.
\end{equation*}
\end{corollary}
\item
Finally we examine the last option $U = H^2(\Omega) \cap H_0^1(\Omega)$, where $\dualprod{Ku}{v}= -\innerL{\Delta u}{ v}$ and $W = L^2(\Omega)$. With the original setting (\ref{A12B1}) we cannot apply the results of Section \ref{sec:Theory}, since $A_1$ is not positive definite.
To overcome this we change the ordering of the variables and equations and obtain the following equivalent form of the optimality conditions:
\begin{equation}
\label{eq:simpleAOCstandard} 
  \widetilde{\mathcal{A}}_\alpha
  \begin{pmatrix} 
    f \\
    w \\
    u   
  \end{pmatrix} = 
  \begin{pmatrix} 
    0\\
    0 \\
    {M} d    
  \end{pmatrix}
  \quad \text{with} \quad
  \widetilde{\mathcal{A}}_\alpha
   = \begin{pmatrix} 
    \alpha M &    M     & 0 \\
    M & 0 & K \\
    0 &    K'    & M 
  \end{pmatrix} .
\end{equation}
Here we view $\widetilde{\mathcal{A}}_\alpha$ as a block operator of the form (\ref{eq:ADefNN}) for $n = 3$ with 
\begin{equation*}
A_1 = \alpha M, \quad A_2 = 0, \quad A_3 = M, \quad B_1 =M, \quad \text{and} \quad B_2=K'.
\end{equation*}
The corresponding Schur complement components are given by
\begin{equation*}
S_1 = \alpha M, \quad S_2=\frac{1}{\alpha}M, \quad \text{and} \quad
S_3 = M+\alpha K' M^{-1}K.
\end{equation*}
As before we have the following alternative representation of $S_3$:
\[
  S_3 = M+\alpha \, B,
\]
with the biharmonic operator, given by (\ref{biharm}).
It is obvious that $S_1$ and $S_2$ are positive definite. We are left with showing that $S_3$ is positive definite, which follows from Lemma \ref{lemma:S3SPD}, since $K$ and $S_3$ in this cases are identical to $K'$ and $\alpha \, S_2$ from the previous case. So, finally we obtain the following result analogously to Corollary \ref{corollary:condNumberOCObservation}:

\begin{corollary}
Under the assumptions of Theorem \ref{theo:convexGrisvard}, the operator $\widetilde{\mathcal{A}}_\alpha$ in Equation (\ref{eq:simpleAOCstandard}) is an isomorphism from $\mathbf{X} = \lTwo\times \lTwo \times \left(H^2\left(\Omega\right)\cap H^1_0\left(\Omega\right)\right)$ to its dual space with respect to the norm in $\mathbf{X}$ given by
\begin{equation*}
\normTwo{\left(f,w,u\right)} = \normTwo{f}_F + \normTwo{w}_W + \normTwo{u}_U
\end{equation*}
with
\begin{equation*}
\normTwo{f}_{F} = \alpha \normTwo{f}_{\lTwo},\quad
\normTwo{w}_{W} = \frac1\alpha \normTwo{w}_{\lTwo},\quad
\normTwo{u}_{U} = \alpha \normTwo{u}_{\lTwo}+\normTwo{\Delta u}_{\lTwo}  .
\end{equation*}
Furthermore, the following condition number estimate holds:
\begin{equation*}
\kappa\left(\mathcal{S}^{-1}_\alpha \widetilde{\mathcal{A}}_\alpha\right) \leq \frac{\cos\left({\pi}/{7}\right)}{\sin\left({\pi}/{14}\right)}\approx4.05
\quad \text{with} \quad 
\mathcal{S}_\alpha 
  = \begin{pmatrix} 
      \alpha M & & \\
      & \frac{1}{\alpha} \, M & \\
      & & M+\alpha \, B 
    \end{pmatrix}.
\end{equation*}
\end{corollary}
\end{enumerate}

The characteristic properties of the model problem of this subsection are:
\begin{itemize}
\item distributed observation: This refers to the first term in the objective functional, where the state $u$ is compared to the given data on the whole domain $\Omega$, and
\item distributed control: The state $u$ is controlled by $f$, which is allowed to live on the whole domain $\Omega$. 
\end{itemize}

Alternatively, the comparison with given data might be done on a set $\Omega_o$ different from $\Omega$, which is called limited observation. 
Similarly, the control might live on a set $\Omega_c$ different from $\Omega$, which is called limited control. 
In the next two subsections, we will see that the results based on the very weak form of the state equation and on the strong form of the state equation of the state equation  can be extended to problems with distributed observation and limited control and to problems with distributed control and limited observation, respectively. For simplicity we will focus on model problems with $\Omega_o = \partial \Omega$ or $\Omega_c = \partial \Omega$.

\subsection{Distributed observation and limited control}
\label{subsec:limitedCntl}

We consider the following variation of the model problem from Subsection \ref{subsec:distrObsdistrCntl}
as a model problem for distributed observation and limited control:

Find $u \in U$ and $f \in F = L^2(\partial\Omega)$ which minimizes the objective functional
  \begin{align*}
    \frac12 \normTwo{u-d}_{\lTwo} + \frac{\alpha}{2}\normTwo{f}_{\lTwoB}  
  \end{align*}
subject to the constraint
  \begin{align*}
    -\Delta u &= 0\quad \text{in} \quad \Omega,\\
     u &= f \quad \text{on} \quad \partial\Omega,
  \end{align*}
where $d\in U = \lTwo$ and $\alpha >0$ are given data.

This model problem and error estimates for a finite element discretization are analyzed in \cite{may2013error} for convex domains $\Omega$. As in 
\cite{may2013error} we consider the very weak form of the state equation:
   \begin{align*}
\innerL{ u}{-\Delta w} +\inner{f}{\partial_{n} w }_{\lTwoB} &= 0,\quad \foralls w\in W,
  \end{align*}
with $u \in U = L^2(\Omega)$ and $W = H^2\left(\Omega\right)\cap H^1_0\left(\Omega\right)$. Here $\partial_n w$ denotes the normal derivative of $w$ on $\partial \Omega$.
Contrary to  \cite{may2013error} we do not assume that $\Omega$ is convex. 
See also \cite{berggren2004approximations} for another version of a very weak formulation which coincides with the formulation from above for convex domains.

Analogous to Subsection \ref{subsec:distrObsdistrCntl} the optimality system can be derived and reads:
\begin{problem}
\label{prob:simpleAOCKKTLC}
Find $\left(f,w,u\right)\in F\times W \times U$  such that 
\begin{equation}
  \mathcal{A}_\alpha
  \begin{pmatrix} 
    u \\
    f \\
    w   
  \end{pmatrix} = 
  \begin{pmatrix} 
    M d \\
    0 \\
    0
  \end{pmatrix} \quad \text{with}\quad
  \mathcal{A}_\alpha=  \begin{pmatrix} 
    M &    0  & K' \\
    0 & \alpha M_{\partial}    & N \\
    K        & N'    &  0
  \end{pmatrix},
\end{equation}
where
\begin{align*}
  \dualprod{M y}{z} &= \inner{y}{z}_{\lTwo}, &
  \dualprod{K u}{w} &= -\inner{u}{\Delta w}_{\lTwo}, \\
  \dualprod{M_\partial f}{g} &= \inner{f}{g}_{\lTwoB}, &
  \dualprod{N w}{f} &= \inner{\partial_n w}{f}_{\lTwoB},
\end{align*}
and
$U = L^2\left(\Omega\right)$, $F = \lTwoB$, and $W=H^2\left(\Omega\right)\cap H^1_0\left(\Omega\right)$.
\end{problem}

Using similar arguments as for Problem \ref{prob:simpleAOCKKTstadard} with the very weak formulation of the state equation, we obtain the following result: 

\begin{corollary}
\label{corollary:condNumberOCControl}
Under the same assumptions of Theorem \ref{theo:convexGrisvard}, the operator $\mathcal{A}_\alpha$ in Problem \ref{prob:simpleAOCKKTLC} is an  isomorphism between $\mathbf{X}=\lTwo \times \lTwoB \times H^2\left(\Omega\right)\cap H^1_0\left(\Omega\right)$ and its dual space with respect to the norm in $\mathbf{X}$ given by 
\begin{align*}
\normTwo{\left(u,f,w\right)} = \normTwo{u}_U + \normTwo{f}_F + \normTwo{w}_W
\end{align*}
with
\begin{equation*}
\normTwo{u}_{U} = \normTwo{u}_{\lTwo},\quad
\normTwo{f}_{F} = \alpha \normTwo{f}_{\lTwoB},\quad
\normTwo{w}_{W} = \normTwo{\Delta w}_{\lTwo}+ \frac{1}{\alpha}\normTwo{\partial_{n}w}_{\lTwoB}  .
\end{equation*}
Furthermore, the following condition number estimate holds 
\begin{equation*}
\kappa\left(\mathcal{S}^{-1}_\alpha\mathcal{A}_\alpha\right) \leq \frac{\cos\left({\pi}/{5}\right)}{\sin\left({\pi}/{10}\right)} \approx 2.62
\quad \text{with} \quad 
{\mathcal{S}}_\alpha 
  = \begin{pmatrix} 
       M & & \\
      & \alpha \, M_\partial & \\
      & &  \frac{1}{\alpha}\, K_{\partial} + B
    \end{pmatrix}
\end{equation*}
where
\begin{align*}
  \dualprod{K_\partial y}{z} &= \inner{\partial_n y}{\partial_n z}_{\lTwoB}.
\end{align*}
\end{corollary}

\subsection{Distributed control and limited observation} 
\label{subsec:limitedObs}

Finally we consider a model problem with distributed control and limited observation: 

Find $u \in U$ and $f \in F = L^2(\Omega)$ which minimizes the objective functional
  \begin{align*}
    \frac12 \normTwo{\partial_n u-d}_{\lTwoB} + \frac{\alpha}{2}\normTwo{f}_{\lTwo}  
  \end{align*}
subject to the constraint
  \begin{align*}
    -\Delta u + f &= 0\quad \text{in} \quad \Omega,\\
     u &= 0 \quad \text{on} \quad \partial\Omega,
  \end{align*}
where $d\in \lTwoB$, $\alpha > 0$ are given data.

Robust preconditioners for this problem were first analyzed in \cite{KAMmagneOC}. As in \cite{KAMmagneOC} the strong form of the state equation is used: $u \in U = H^2\left(\Omega\right)\cap H^1_0\left(\Omega\right)$ and
\begin{equation*}
  - \inner{\Delta u}{w}_{\lTwo} + \inner{f}{w}_{\lTwo} = 0 \quad \foralls w \in W = \lTwo.
\end{equation*}

Following the same procedure as for Problem \ref{prob:simpleAOCKKTstadard} with $U = H^2\left(\Omega\right)\cap H^1_0\left(\Omega\right)$
we obtain the (reordered) optimality system.

\begin{problem}
\label{prob:simpleAOCKKTLO}
Find $\left(f,w,u\right)\in W\times W \times U$ such that 
\begin{equation*}
\widetilde{\mathcal{A}}_\alpha
  \begin{pmatrix} 
    f \\
    w \\
    u   
  \end{pmatrix} = 
  \begin{pmatrix} 
    0 \\
    0 \\
    N' d   
  \end{pmatrix}\quad \text{with} \quad 
\widetilde{\mathcal{A}}_\alpha=
  \begin{pmatrix} 
    \alpha M &    M  & 0 \\
    M& 0    & K \\
    0        & K'    & K_\partial
  \end{pmatrix},
\end{equation*}
where
\begin{align*}
  \dualprod{K u}{w} &= -\inner{\Delta u}{ w}_{\lTwo}, &
  \dualprod{N' d}{v} = \dualprod{N v}{d} &= \inner{\partial_n v}{d}_{\lTwoB},
\end{align*}
and $W = F = L^2\left(\Omega\right)$, and $U=H^2\left(\Omega\right)\cap H^1_0\left(\Omega\right)$.
\end{problem}
Using similar arguments as for Problem \ref{prob:simpleAOCKKTstadard} with $U = H^2\left(\Omega\right)\cap H^1_0\left(\Omega\right)$, we obtain the following result:

\begin{corollary}
\label{cor:distrcntllimitobs}
Under the same  assumptions of Theorem \ref{theo:convexGrisvard}, the operator  $\widetilde{\mathcal{A}}_\alpha$ in Problem \ref{prob:simpleAOCKKTLO} is an isomorphism between $\mathbf{X}= \lTwo\times \lTwo \times H^2\left(\Omega\right)\cap H^1_0\left(\Omega\right)$ and its dual space, with respect to the norm in $\mathbf{X}$ given by 
\begin{align*}
\normTwo{\left(f,w,u\right)} = \normTwo{f}_F + \normTwo{w}_W + \normTwo{u}_U
\end{align*}
with
\begin{align*}
\normTwo{f}_{F} &= \alpha \normTwo{f}_{\lTwo}, \quad
\normTwo{w}_{W} &= \frac1\alpha \normTwo{w}_{\lTwo},\quad 
\normTwo{u}_{U} &= \normTwo{\partial_{n}u}_{\lTwoB} +\alpha \normTwo{\Delta u}_{\lTwo}.
\end{align*}
Furthermore, the following condition number estimate holds 
\begin{equation*}
\kappa\left(\mathcal{S}^{-1}_\alpha\widetilde{\mathcal{A}}_\alpha\right) \leq \frac{\cos\left({\pi}/{7}\right)}{\sin\left({\pi}/{14}\right)}\approx4.05
\quad \text{with} \quad 
\mathcal{S}_\alpha 
  = \begin{pmatrix} 
       \alpha \,M & & \\
      & \frac{1}{\alpha} \, M & \\
      & &    K_{\partial}+\alpha \, B
    \end{pmatrix}.
\end{equation*}
\end{corollary}

Corollary \ref{cor:distrcntllimitobs} with the preconditioner $\mathcal{S}_\alpha$ was originally proven in \cite{KAMmagneOC}, which was the main motivation for this article. In \cite{KAMmagneOC} 
convexity of $\Omega$ was required.

%% file: Discretization.tex
\section{Preconditioners for discretized optimality systems}
\label{sec:Discretization}

So far we have only addressed optimality systems on the continuous level. In this section we discuss the discretization of optimality systems and efficient preconditioners for the discretized problems.
We will focus on Problem \ref{prob:simpleAOCKKTLO}. The same approach also applies to Problems \ref{prob:simpleAOCKKTstadard} and \ref{prob:simpleAOCKKTLC}.

Let $U_h$ and $W_h$ be conforming finite-dimensional approximation spaces for Problem \ref{prob:simpleAOCKKTLO}; that is,
\[
  U_h \subset H^2(\Omega)\cap H_0^1(\Omega), \quad W_h \subset L^2(\Omega). 
\]
Applying Galerkin's principle to Problem \ref{prob:simpleAOCKKTLO} leads to the following problem:

\begin{problem}
\label{prob:simpleKKTDisLO}
  Find $(f_h,w_h,u_h) \in W_h \times W_h \times U_h$ such that
  \begin{equation}
    \label{eq:simpleA} 
    \widetilde{\mathcal{A}}_{\alpha,h}
  \begin{pmatrix} 
    \underline{f}_h \\
    \underline{w}_h \\
    \underline{u}_h   
  \end{pmatrix} = 
  \begin{pmatrix} 
    0 \\
    0 \\
    \underline{d}_h 
  \end{pmatrix}
  \quad \text{with} \quad
   \widetilde{\mathcal{A}}_{\alpha,h}
     = 
    \begin{pmatrix} 
    \alpha M_h &  M_h    & 0 \\
    M_h        & 0  & K_h \\
    0          & K_h^T  & K_{\partial,h}
  \end{pmatrix} ,
  \end{equation}
where $M_h$, $K_h$, $K_{\partial,h}$ are the matrix representations of linear operators $M$, $K$, $K_\partial$ on $W_h$, $U_h$, $U_h$ relative to chosen bases in these spaces, respectively, and $\underline{f}_h$, $\underline{w}_h$, $\underline{u}_h$, $\underline{d}_h$ are the corresponding vector representations of $f_h$, $w_h$, $u_h$, $N' d$.
\end{problem}

Motivated by Corollary \ref{cor:distrcntllimitobs} we propose the following preconditioner for (\ref{eq:simpleA}):
\begin{equation}
  \label{eq:simplePreOCdis} 
  \mathcal{S}_{\alpha,h} =
  \begin{pmatrix} 
    \alpha M_h &    0  & 0 \\
    0 & \frac{1}{\alpha} M_h & 0\\
    0        & 0    & K_{\partial,h} + \alpha \, B_h
  \end{pmatrix},
\end{equation}
where $B_h$ is given by
\begin{equation}
\label{eq:biharmonic}
\dualprod{B_h \underline{u}_h}{\underline{v}_h} = \innerL{\Delta u_h}{\Delta v_h}, \quad \foralls u_h, v_h \in U_h.
\end{equation}
The operator $\mathcal{S}_{\alpha}$ is self-adjoint and positive definite. Therefore, the preconditioner $\mathcal{S}_{\alpha,h}$ is symmetric and positive definite, since it is obtained by Galerkin's principle. Moreover, the preconditioner $\mathcal{S}_{\alpha,h}$ is a sparse matrix, provided the underlying bases of $U_h$ and $W_h$ consist of functions with local support, which we assume from now on.
The application of the preconditioner within a preconditioned Krylov subspace method requires to solve linear systems of the form
\begin{equation} \label{Swr}
  \mathcal{S}_{\alpha,h} \, \underline{\mathbf{w}}_h = \underline{\mathbf{r}}_h
\end{equation}
for given vectors $\underline{\mathbf{r}}_h$.
Since $\mathcal{S}_{\alpha,h}$ is a sparse matrix, sparse direct solvers can be used for efficiently solving (\ref{Swr}), which is the preferred choice in this paper, see Chapter \ref{sec:Results}. Alternatively, one could also consider geometric or algebraic multigrid methods for solving (\ref{Swr}). For more information about multigrid methods, see \cite{trottenberg2000multigrid}.

Observe that, in general, the preconditioner $\mathcal{S}_{\alpha,h}$ introduced above is different from the Schur complement preconditioner
\begin{equation} \label{theoretPrecond}
  \mathcal{S}(\widetilde{\mathcal{A}}_{\alpha,h}) = \begin{pmatrix} 
    \alpha M_h &    0  & 0 \\
    0 & \frac{1}{\alpha} M_h & 0\\
    0        & 0    & K_{\partial,h} + \alpha \, K_h^{T}M^{-1}_h K_h
  \end{pmatrix},
\end{equation}
as introduced in (\ref{eq:SchurDefNN}). 
Therefore, in general, the condition number estimates derived in Section \ref{sec:Theory} do not hold for $\mathcal{S}_{\alpha,h}$. There is one exception from this rule provided by the next lemma, which is due to \cite{KAMmagneOC}.
We include the short proof for completeness.

\begin{lemma}
\label{lemma:spaceCompLO}
Let $U_h$ and $W_h$ be conforming discretization spaces to Problem \ref{prob:simpleKKTDisLO} 
with \begin{equation} \label{suff1}
\Delta U_h \subset W_h.
\end{equation}
Then we have
\begin{equation*} 
    K_h^{T}M^{-1}_h K_h = B_h.
\end{equation*}
\end{lemma}

\begin{proof}\unskip
We have 
\begin{align*}
\dualprod{K_h^{T}M^{-1}_h K_h \underline{u}_h}{\underline{u}_h}
  & = \sup_{0 \neq w_h\in W_h}\frac{\dualprod{K_h \underline{u}_h}{\underline{w}_h}^{2}}{\dualprod{M_h\underline{w}_h}{\underline{w}_h}} 
    = \sup_{0 \neq w_h\in W_h}\frac{\inner{-\Delta u_h}{w_h}^{2}_{\lTwoB}}{\inner{w_h}{w_h}_{L^2\left(\Omega\right)}} \\
  & = \norm{\Delta u_h}_{\lTwo}^2 = \dualprod{B_h \underline{u}_h}{\underline{u}_h}.
\end{align*}
Since both $K_h^{T}M^{-1}_h K_h$ and $B_h$ are symmetric matrices, equality follows.
\end{proof}

So, under the assumptions of Lemma \ref{lemma:spaceCompLO} we have $\mathcal{S}_{\alpha,h} = \mathcal{S}(\widetilde{\mathcal{A}}_{\alpha,h})$, and, therefore, it follows that
\begin{equation} \label{robustS}
  \kappa(\mathcal{S}_{\alpha,h}^{-1} \, \widetilde{\mathcal{A}}_{\alpha,h}) \le \frac{\cos(\pi/7)}{\sin(\pi/14)} \approx 4.05,
\end{equation}
showing that $\mathcal{S}_{\alpha,h}$ is a robust preconditioner in $\alpha$ and in $h$. 

\begin{remark}
In case that Condition (\ref{suff1}) does not hold, the matrix $K_h^{T}M^{-1}_h K_h$ must be expected to be dense. This makes the application of the Schur complement preconditioner $\mathcal{S}(\widetilde{\mathcal{A}}_{\alpha,h})$ computationally too expensive, while $\mathcal{S}_{\alpha,h}$ is always sparse.
\end{remark}

While $\mathcal{S}_{\alpha,h}$ is always positive definite, the Schur complement preconditioner $\mathcal{S}(\widetilde{\mathcal{A}}_{\alpha,h})$ is symmetric but, in general, only positive semi-definite. However, a simple and mild condition guarantees the definiteness: 
\begin{lemma}
\label{lemma:spaceCompLOb}
Let $U_h$ and $W_h$ be conforming discretization spaces to Problem \ref{prob:simpleKKTDisLO} 
with 
\begin{equation} \label{suff2}
  U_h \subset W_h.
\end{equation}
Then the matrix $K_{\partial,h} + \alpha K_h^{T}M^{-1}_h K_h$ is symmetric and positive definite.
\end{lemma}

\begin{proof}
If (\ref{suff2}) holds, then it follows that
\begin{align*}
 \dualprod{K_h^{T}M^{-1}_h K_h \underline{u}_h}{\underline{u}_h} = \sup_{0 \neq w_h\in W_h}\frac{\inner{-\Delta u_h}{w_h}^{2}_{L^2\left(\Omega\right)}}{\inner{w_h}{w_h}_{L^2\left(\Omega\right)}} 
  \ge \frac{\inner{-\Delta u_h}{u_h}^{2}_{L^2\left(\Omega\right)}}{\inner{u_h}{u_h}_{L^2\left(\Omega\right)}} ,
\end{align*}
by choosing $w_h = u_h \in W_h$.
Therefore, if $\underline{u}_h$ is in the kernel of 
$K_h^{T}M^{-1}_h K_h$, then  $u_h \in U_h \subset H^2(\Omega) \cap H_0^1(\Omega)$ 
and $\inner{\nabla u_h}{\nabla u_h}_{\lTwo} = \inner{-\Delta u_h}{u_h}^{2}_{L^2\left(\Omega\right)} = 0$, which imply $u_h = 0$.
\end{proof}

This lemma shows the importance of Condition (\ref{suff2}), which we will adopt as a condition for our choice of $U_h$ and $W_h$, see below. Additionally, it allows us to compare the practical preconditioner $\mathcal{S}_{\alpha,h}$ with the theoretical Schur complement preconditioner $\mathcal{S}(\widetilde{\mathcal{A}}_{\alpha,h})$, which would guarantee the derived uniform bound of the condition number but is computationally too costly.

Observe that it is required that $U_h \subset H^2(\Omega)\cap H_0^1(\Omega)$. In order to meet this condition $C^1$ finite element spaces were proposed for $U_h$ in \cite{KAMmagneOC}. In particular, the Bogner-Fox-Schmit element on a rectangular mesh was used. 
Here we advocate instead for spline spaces of sufficiently smooth functions as provided in Isogeometric Analysis. For the purpose of this paper we restrict ourselves to a simple version of such approximation spaces, which are shortly described now.
Let $\hat{S}^{p}_{k, \ell}$ be the space of spline functions on the unit interval $(0,1)$ which are $k$-times continuously differentiable and piecewise polynomial of degree $p$ on a mesh of mesh size $2^{-\ell}$ which is obtained by $\ell$ uniform refinements of $(0,1)$. The value $k = -1$ is used for discontinuous spline functions. On 
$(0,1)^d$ we use the corresponding tensor-product spline space, which, for simplicity, is again denoted by $\hat{S}^{p}_{k,\ell}$. It will be always clear from the context what the actual space dimension $d$ is.
It is assumed that the physical domain $\Omega$, can be parametrized by a mapping $\mathbf{F}: (0,1)^d\rightarrow \Omega$ with components $\mathbf{F}_i \in \hat{S}^{p}_{k,\ell}$. The discretization space  $S^{p}_{k,\ell}$ on the domain $\Omega$ is defined by
\begin{equation*} 
S^{p}_{k,\ell} := 
\left\lbrace f \circ \mathbf{F}^{-1} : f \in 
\hat{S}^{p}_{k,\ell}\right\rbrace .
\end{equation*}
All information on this discretization space is summarized by the triple $h = (p,k,\ell)$.
See the monograph \cite{cottrell2009isogeometric} for more details and more sophisticated discretization spaces in IgA.

We proposed the following approximation spaces of equal order:
\[
  U_h = \{ v_h \in S^{p}_{k,\ell} \colon M_{\partial,h} \underline{u}_h = 0 \},
\]
where $M_{\partial,h}$ is the matrix representation of $M_\partial$ on $S^{p}_{k,\ell}$, and
\[
  W_h = S^{p}_{k,\ell}.
\]
For this setting Condition (\ref{suff1}) is not satisfied and the analysis of the proposed preconditioner is not covered by the results of Section \ref{sec:Theory}. Condition (\ref{suff2}) is obviously satisfied  and we will report on promising numerical results in Chapter \ref{sec:Results}.

\begin{remark}
Condition (\ref{suff1}) is rather restrictive. Even if the geometry mapping $\mathbf{F}$ is the identity, the smallest tensor product spline space for $W_h$ for which Condition (\ref{suff1}) holds, is the space $S^{p}_{k-2,\ell}$ if $d \ge 2$. This space has a much higher dimension than the choice $S^{p}_{k,\ell}$ from above without significantly improved approximation properties.
\end{remark}

\begin{remark}
\label{remark:SdisLC}
A completely analogous discussion can be done for the model problems in Subsections \ref{subsec:distrObsdistrCntl} and \ref{subsec:limitedCntl}. For example, 
a sparse preconditioner for the discretized version of Problem \ref{prob:simpleAOCKKTLC}
is given by
\begin{equation*}
{\mathcal{S}}_{\alpha,h} 
  = \begin{pmatrix} 
       M_h & & \\
      & \alpha \, M_{\partial,h} & \\
      & &  \frac{1}{\alpha}\, K_{\partial,h} + B_h
    \end{pmatrix},
\end{equation*}
motivated by Corollary \ref{corollary:condNumberOCControl}.
\end{remark}

%% file: ResultsBoth.tex
\section{Numerical results}
\label{sec:Results}

In this section we present numerical results for two examples of Problem \ref{prob:simpleAOCKKTLO}.

First we consider a two-dimensional example, where the physical domain $\Omega$ is given by its parametrization
$\mathbf{F}\colon (0,1)^2 \rightarrow \mathbb{R}^2$ with
\[
  \mathbf{F}(\mathbf{\xi}) 
  = \begin{pmatrix}
      1+\xi_1 - \xi_1 \xi_2 - \xi_2^2 \\
      2 \xi_1 \xi_2 - \xi_1 \xi_2^2 +2 \xi_2 - \xi_2^2
    \end{pmatrix},
\]
and the prescribed data $d$ are given by $d(\mathbf{x}) = \partial_n (\sin{\left(2\pi x_1\right)} \sin{\left(4\pi x_2\right)}) $ on the boundary $\partial \Omega$.
The domain $\Omega = \mathbf{F}((0,1)^2)$ is a close approximation of a quarter of an annulus, see Figure \ref{fig:2DDomain}. 
\begin{figure}[h]
  \center
  \includegraphics[width=0.45\hsize]{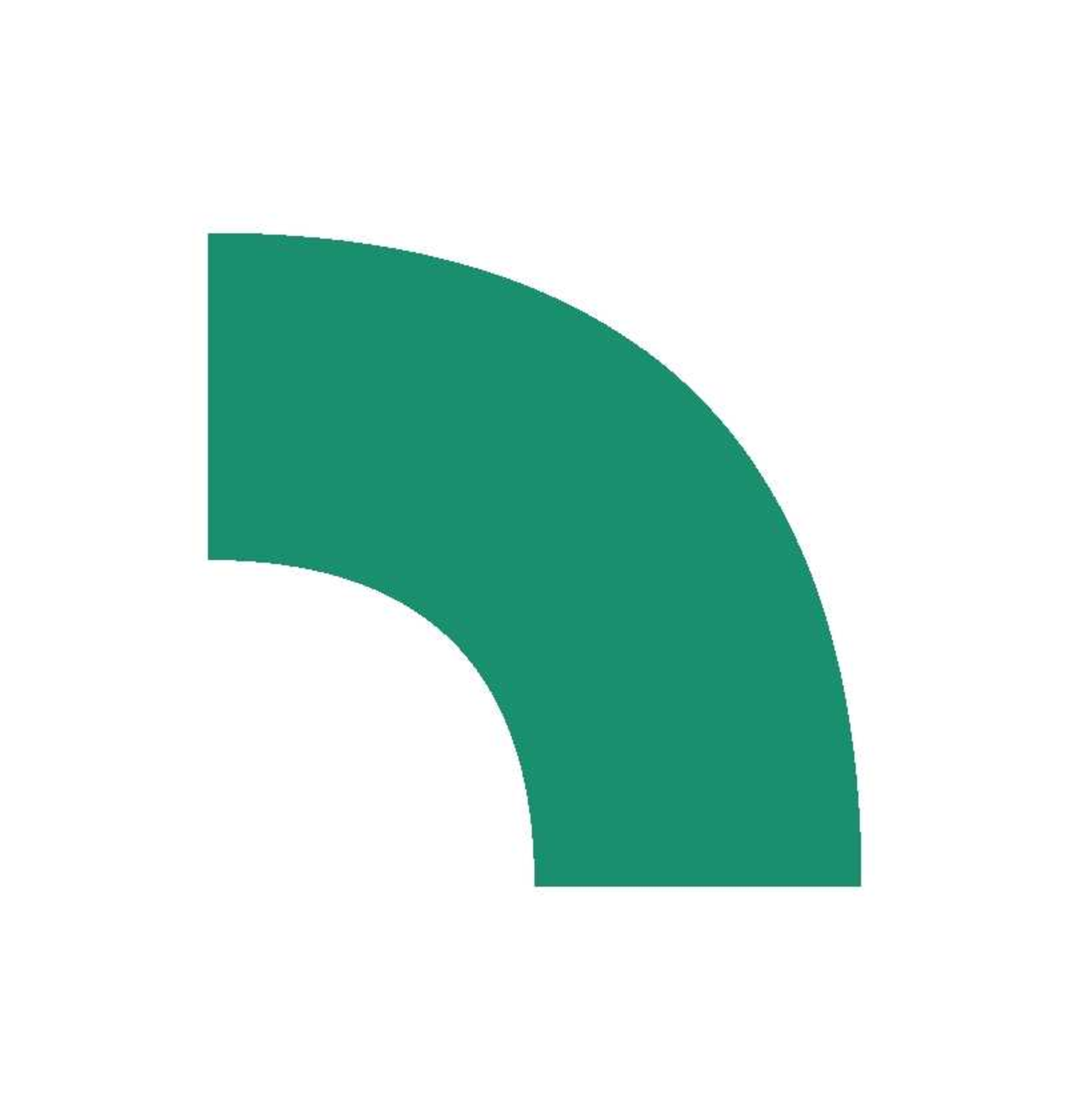}
  \caption{The 2D domain $\Omega$ is an approximation of a quarter of an annulus}
  \label{fig:2DDomain}
\end{figure}

For a fixed polynomial degree $p$ we choose the following discretization spaces of maximal smoothness $k=p-1$:
\[
  U_h = \{ v_h \in S^{p}_{p-1,\ell} \colon M_{\partial,h} \underline{u}_h = 0 \}
  \quad \text{and} \quad W_h = S^{p}_{p-1,\ell}.
\]
The resulting linear system of equations
\[
  \widetilde{\mathcal{A}}_{\alpha,h} \, \underline{\mathbf{x}}_h = \underline{\mathbf{b}}_h
\]
is  solved by using the preconditioned minimal residual method. 
We will present results for the preconditioner $\mathcal{S}_{\alpha,h}$, see (\ref{eq:simplePreOCdis}), and for comparison only, for the Schur complement preconditioner $\mathcal{S}(\widetilde{\mathcal{A}}_{\alpha,h})$, see (\ref{theoretPrecond}).

The iteration starts with the initial guess 0 and stops when
\begin{equation}
 \norm{\underline{\mathbf{r}}_k} \le \epsilon \, \norm{\mathbf{r}_0}
 \quad \text{with} \quad \epsilon = 10^{-8},
\end{equation}
where $\underline{\mathbf{r}}_k$ denotes the residual of the preconditioned problem at $\underline{\mathbf{x}}_k$ and $\norm{.}$ is the Euclidean norm.
All computations are done with the C++ library G+Smo \cite{gismoweb}.

For polynomial degree $p = 2$, Table \ref{t:KKTCanonicalAMAPLO} shows the total number of degrees of freedom (dof) and the number of iterations for different values of the  refinement level $\ell$ and the parameter $\alpha$, when using the Schur complement preconditioner $\mathcal{S}(\widetilde{\mathcal{A}}_{\alpha,h})$. 

\begin{table}[h]
\begin{center} 
  \begin{tabular}{| c | r || c | c | c | c | c | c |}
    \hline
    & & \multicolumn{6}{c|}{$\alpha$} \\ \hline
    $\ell$& \multicolumn{1}{c||}{dof} & 1 & 0.1 & 0.01 & 1e-3 & 1e-5& 1e-7 \\ \hline \hline
    3    & 264    &21&    36&    33&    22&     9&    5 \\  \hline
    4    & 904    &21&    35&    38&    26&     9&    5 \\  \hline
    5    & 3\,336 &21&    35&    35&    29&    10&    5 \\  \hline
    6    & 12\,808&19&    34&    34&    27&     9&    4 \\  \hline
  \end{tabular} %
  \caption{2D example, number of iterations for the preconditioner $\mathcal{S}(\widetilde{\mathcal{A}}_{\alpha,h})$.}
  \label{t:KKTCanonicalAMAPLO}
\end{center}
\end{table}

As predicted from the analysis the number of iterations are bounded uniformly with respect to $\alpha$ and $\ell$.

Table \ref{t:KKTCanonicalBLO} shows the number of iterations when using $\mathcal{S}_{\alpha, h}$. 
\begin{table}[h]
\begin{center}
  \begin{tabular}{| c | r || c | c | c | c | c | c |}
    \hline
    & & \multicolumn{6}{c|}{$\alpha$} \\ \hline
    $\ell$& \multicolumn{1}{c||}{dof} & 1 & 0.1 & 0.01 & 1e-3 & 1e-5& 1e-7 \\ \hline \hline
    3  & 264   &24&    38&    39&    34&    23&    19 \\  \hline
    4  & 904   &25&    38&    41&    36&    22&    18\\  \hline
    5  & 3\,336&25&    38&    40&    34&    22&    17 \\  \hline
    6  &12\,808&25&    38&    39&    31&    19&    13 \\  \hline
  \end{tabular}
  \caption{2D example, number of iterations for the preconditioner $\mathcal{S}_{\alpha, h}$.}
  \label{t:KKTCanonicalBLO}
\end{center}
\end{table}

The numbers in Table \ref{t:KKTCanonicalBLO} are only slightly larger than the numbers in Table \ref{t:KKTCanonicalAMAPLO} for large $\alpha$. For small $\alpha$ some additional iterations are required, nevertheless is appears that method seems to be robust with respect to $\alpha$ and the refinement level $\ell$.

As a second example we consider a three-dimensional variant of the two-dimensional example. The physical domain $\Omega$ is obtained by twisting a cylindrical extension of the two-dimensional domain from the first example. The parametrization is given by the geometry map $\mathbf{F}\colon (0,1)^3 \rightarrow \mathbb{R}^3$ with 
\[
  \mathbf{F}(\mathbf{\xi}) 
  = \begin{pmatrix}
      \frac32 \xi_1 \xi_2^3 \xi_3 - \xi_1 \xi_2^3 -\frac32 \xi_1 \xi_2^2 \xi_3 + \xi_1 +\frac12 \xi_2^3 \xi_3 +\frac12 \xi_2^3 + \frac32 \xi_2^2 \xi_3 -\frac32 \xi_2^2 + 1\\
      \xi_2\left(\xi_1 \xi_2^2 - 3 \xi_1 \xi_2 + 3\xi_1 -\frac12 \xi_2^2 +\frac32 \right)\\
-\xi_2^3 \xi_3 +\frac12 \xi_2^3 + \frac32 \xi_2^2 + \xi_3 
    \end{pmatrix}
\]
and the prescribed data $d$ are given by $d(\mathbf{x}) = \partial_n (\sin{\left(2\pi x_1\right)} \sin{\left(4\pi x_2\right)} \sin{\left(6\pi x_3\right)}) $ on the boundary $\partial \Omega$.
\begin{figure}[h]
  \center
  \includegraphics[width=0.45\hsize]{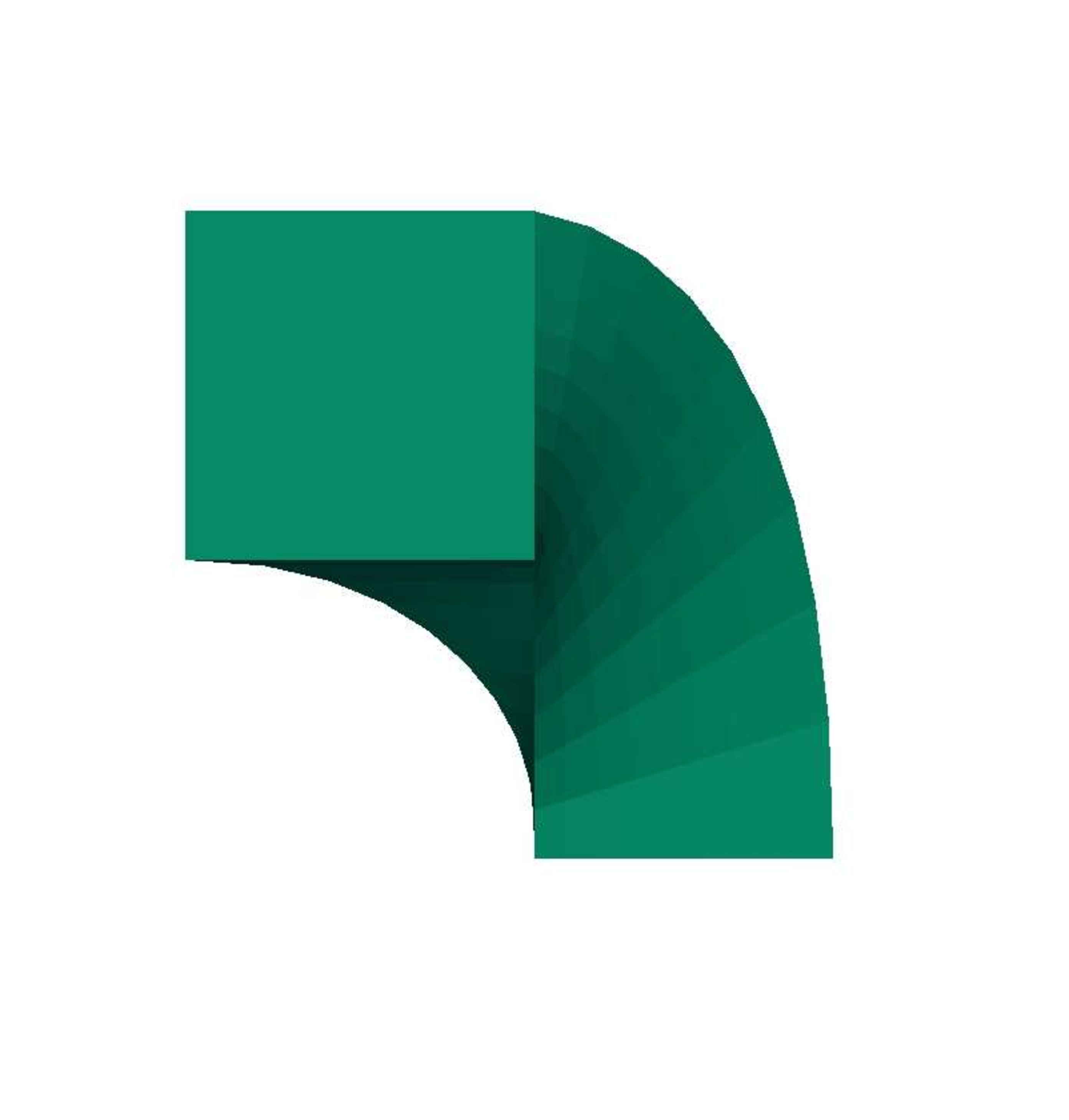}
  \includegraphics[width=0.45\hsize]{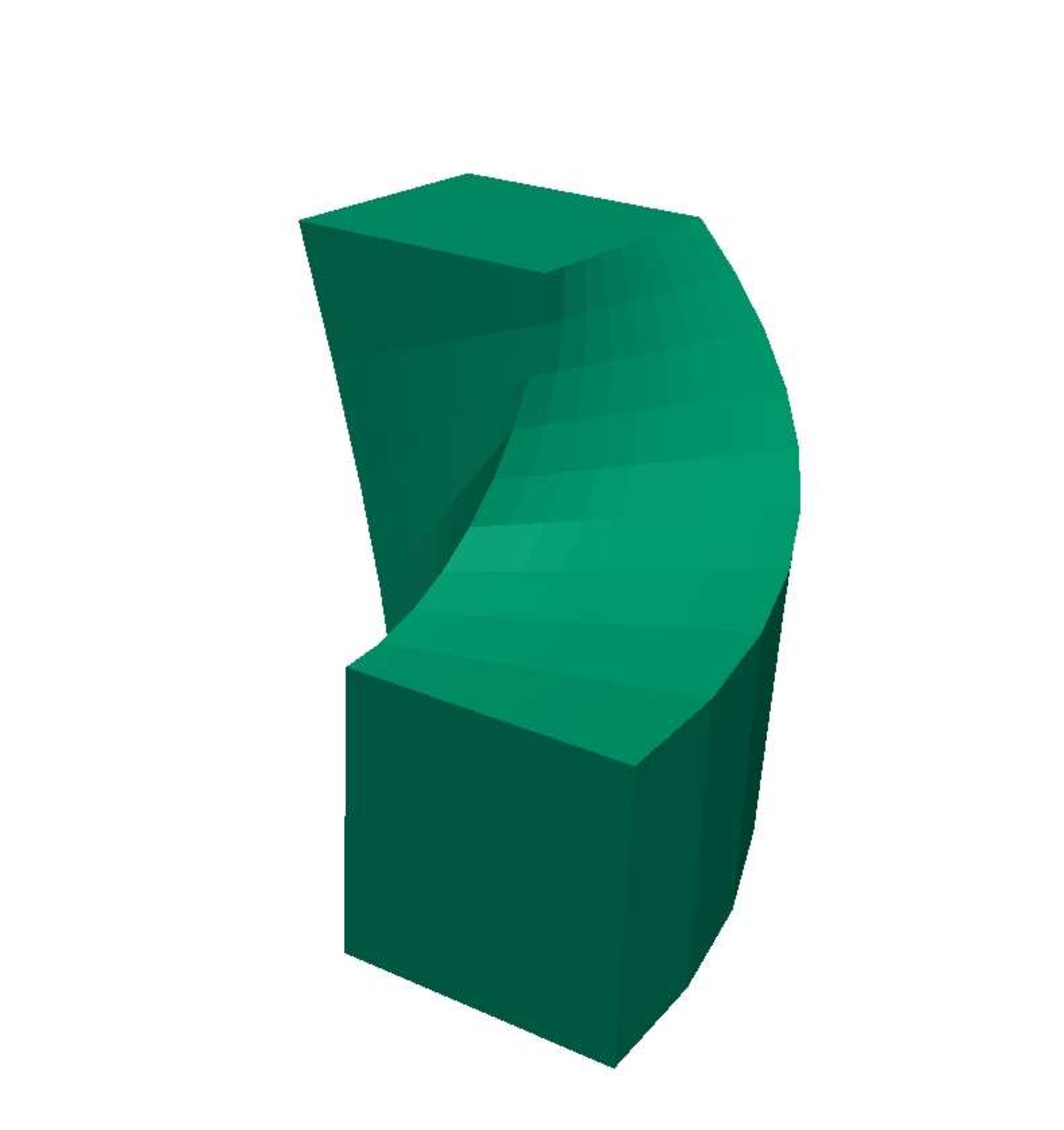}
  \caption{The 3D domain viewed from two different angles}
  \label{fig:3DDomain}
\end{figure}

For polynomial degree $p=3$, Table \ref{t:3DP3} shows the number of iterations for the three-dimensional example, see Figure \ref{fig:3DDomain}, 
using the preconditioner $\mathcal{S}_{\alpha, h}$.

\begin{table}[h]
\begin{center}
  \begin{tabular}{| c | r || c | c | c | c | c | c |}
    \hline
    & & \multicolumn{6}{c|}{$\alpha$} \\ \hline
    $\ell$& \multicolumn{1}{c||}{dof} & 1 & 0.1 & 0.01 & 1e-3 & 1e-5& 1e-7 \\ \hline \hline
    2     & 811     &20 & 35  & 41  & 32 & 19 & 16   \\  \hline
    3     & 3\,391  &23 & 35  & 43  & 40 & 22 & 18   \\  \hline
    4     & 18\,631 &23 & 35  & 43  & 37 & 22 & 17   \\  \hline
    5     & 121\,687&19 & 33  & 38  & 34 & 20 & 13   \\  \hline
  \end{tabular}
  \caption{3D example, number of iterations for the preconditioner $\mathcal{S}_{\alpha, h}$.} 
  \label{t:3DP3}
\end{center}
\end{table}

The number of iterations for the 3D example are similar to their 2D counterpart. 

%% file: Discussion.tex
\section{Conclusions}
\label{sec:Discussion}

Two main results have been shown: new existence results for optimality systems in Hilbert spaces and sharp condition number estimates.
Typical applications for the new existence results are  model problems from optimal control problems with second-order elliptic state equations.
For boundary observation and distributed control the existence of the optimal state in $H^2(\Omega)$ follows for polygonal/polyhedral domains without additional convexity assumptions, although the state equation alone does not guarantee the existence of a solution in $H^2(\Omega)$ if the right-hand side lies in $L^2(\Omega)$. For this class of problems, which initially are seen as classical saddle point problems, it turned out that the reformulation as multiple saddle point problems is beneficial.
Similarly, for distributed observation and boundary control the existence of the optimal Lagrangian in $H^2(\Omega)$ follows for polygonal/polyhedral domains without convexity assumptions. 
These new existence results were obtained by replacing the standard weak formulation of second-order problem by a strong or a very weak formulation depending on the type of optimal control problems.

The new sharp condition number estimates for multiple saddle point problems are to be seen as extensions of well-known sharp bounds for standard saddle point problems. The analysis of saddle point problems in function spaces motivates the construction of sparse preconditioners for discretized optimality systems. The interpretation of standard saddle point problems with primal and dual variables as multiple saddle point problems with possibly more than two types of variables allows the construction of preconditioners based on Schur complements for a wider class of problems.

And, finally, the required discretization spaces of higher smoothness can be handled with techniques from Isogeometric Analysis, which opens the doors to possible extensions to optimal control problems with other classes of state equations like biharmonic equations.

%% file: Appendix.tex
\section{Appendix}

The Chebyshev Polynomials of second kind are defined by the recurrence relation
\[
  U_{0}(x) = 1, \quad
  U_{1}(x) = 2x, \quad 
  U_{i+1}(x) = 2xU_{i}(x)-U_{i-1}(x)
  \quad \text{for} \ i \ge 1.
\]
Their closed form representation is given by
\begin{equation} \label{Urep}
  U_j(\cos \theta) = \frac{\sin\left(\left(j+1\right)\theta\right)}{\sin\left(\theta\right)},
\end{equation}
see \cite{rivlin1990chebyshev}.

It immediately follows that the polynomials $P_j(x) := U_j(x/2)$ satisfy the related recurrence relation
\[
P_{0}(x) = 1, \quad 
P_{1}(x) = x, \quad
P_{i+1}(x) = xP_{i}(x)-P_{i-1}(x)
\quad \text{for} \ i \ge 1,
\]
which shows that the polynomials $P_j(x)$ coincide with the polynomials used in the proof of Theorem \ref{theo:CBound}. 
Analogously it follows that the polynomials $\bar{P}_j(x) := P_j(x) - P_{j-1}(x) = U_j(x/2) - U_{j-1}(x/2)$ satisfy the related recurrence relation
\[
\bar{P}_{0}(x) = 1, \quad
\bar{P}_{1}(x) = x-1, \quad
\bar{P}_{i+1}(x) = x\bar{P}_{i}(x)-\bar{P}_{i-1}(x)
\quad \text{for} \ i \ge 1, 
\]
which shows that the polynomials $\bar{P}_j(x)$ coincide with the polynomials used in the proof of Theorems \ref{theo:CBound} and \ref{theo:exactEigV}.

In the next lemma properties of the roots of the polynomials $\bar{P}_j(x)$ are collected, which were used in these theorems.

\begin{lemma} \ \label{lem:CheByRootsandRealProblem} 
\begin{enumerate}
\item
The roots of the polynomial $\bar{P}_j$ are given by
\begin{equation*}
x^{i}_{j} = 2 \, \cos\left(\frac{2i-1}{2j+1}\pi\right), \quad \text{for } i= 1,\ldots,j.
\end{equation*}
\item
For fixed $j$, the root of largest modulus is $x_j^1$. Moreover,
\[
  x_j^1 > 1 \quad \text{and} \quad \bar{P}_i(x_j^1) > 0 \quad \text{for all} \ i = 0,1,\ldots,j-1.
\]
\item
For fixed $j$, the root of smallest modulus $x_j^*$ is given by $x_j^* = x_j^{i^*}$ with $i^* = [j/2]+1$, where $[y]$ denotes the largest integer less or equal to $y$, Moreover,
\[
  \left|x_j^*\right| = 2 \sin \left( \frac{1}{2(2j+1)} \pi \right).
\]
\end{enumerate}
\end{lemma}
\begin{proof}\unskip
  From (\ref{Urep})
  we obtain
\begin{align*}
\bar{P}_j(2\cos \theta) = \frac{1}{\sin \theta} \big(\sin\left(\left(j+1\right)\theta\right)-\sin\left(j\theta\right)\big)
  = \frac{2\sin(\theta/2)}{\sin \theta} \, \cos\left(\frac{2j+1}{2}\theta\right).
\end{align*}
Then the roots of $\bar{P}_j$ directly follow from the known zeros $\frac{2i-1}{2} \pi$ of $\cos(x)$. For fixed $j$, $x_j^i$ is a decreasing sequence in $i$, for which the rest of the lemma can deduced by elementary calculations.
\end{proof}

In the proof of Theorem \ref{theo:CInvBound} a sequence of matrices $Q_j$ is introduced, whose spectral norms is needed.
It is easy to verify that
\begin{equation}
\label{eq:PnMatDef}
 Q^{-1}_{j} = 
\begin{pmatrix} 
      1  & -1  &   & &\\
      -1 & 0  & 1 &  &\\
         & 1  & 0& \ddots&\\
         &  &\ddots & \ddots & (-1)^{j-1} \\
          &    &   & (-1)^{j-1}  & 0
\end{pmatrix} .
\end{equation}
By Laplace's formula one sees that the polynomials $\det (\lambda \, I - Q_n^{-1})$ satisfy the same recurrence relation as the polynomials $\bar{P}_j(\lambda)$, and, therefore, we have
\[
  \det (\lambda \, I - Q_n^{-1}) = \bar{P}_j(\lambda).
\]
Hence, with the notation from above it follows that
\[
  \|Q_j\| = \frac{1}{\left|x_j^*\right|} = \frac{1}{2 \sin \left( \frac{1}{2(2j+1)} \pi \right)},
\]
which was used for the calculating $\|\mathcal{M}^{-1}\|_{\mathcal{L}(\mathbf{X},\mathbf{X})}$ in Theorem \ref{theo:CInvBound}.